\documentclass[12pt,oneside]{amsart}
\usepackage{calc,geometry, verbatim, graphicx, amssymb, color, mathabx, mathtools, enumitem, bm, mathrsfs, amsmath, amsthm, ulem, xspace}
\usepackage[usenames,dvipsnames,svgnames,table]{xcolor}
\usepackage[pdftex,bookmarks,colorlinks,breaklinks]{hyperref}  
\hypersetup{linkcolor=blue,citecolor=red,filecolor=dullmagenta,urlcolor=darkblue} 
\usepackage{multirow}
\usepackage{adjustbox}
\newtheorem{theorem}{Theorem}[section]

\newtheorem{lemma}[theorem]{Lemma}

\newtheorem{remark}[theorem]{Remark}
\newtheorem{example}[theorem]{Example}
\newtheorem{question}{Problem} 
\newcommand{\Mod}[1]{\ (\mathrm{mod}\ #1)}

\newcommand\mult{\operatorname{\textup{{\fontfamily{ptm}\selectfont mult}}}}
\newcommand\dg{\operatorname{\textup{{\fontfamily{ptm}\selectfont deg}}}}
\newcommand\frc{\operatorname{\textup{{\fontfamily{ptm}\selectfont frac}}}}

\newcommand{\cc}{\mathcal}

\newcommand\rounddown[1]{\left\lfloor#1\right\rfloor}

      \makeatletter
      \def\@setcopyright{}
      \def\serieslogo@{}
      \makeatother      

\begin{document}
   \author{Amin Bahmanian}
   \address{Department of Mathematics,
  Illinois State University, Normal, IL USA 61790-4520}

   \author{Sadegheh Haghshenas}
   \address{Department of Pathology and Laboratory Medicine,
  Western University, London, ON, Canada, N6A 5C1}
  \email{shaghsh@uwo.ca}

\title[On Regular Set Systems Containing Regular Subsystems]{On Regular Set Systems Containing Regular Subsystems}
   \begin{abstract}  
Let $X,Y$ be finite sets, $r,s,h, \lambda \in \mathbb{N}$ with $s\geq r, X\subsetneq Y$. By $\lambda \binom{X}{h}$ we mean the collection of all $h$-subsets of $X$ where each subset occurs $\lambda$ times. 
 A coloring of $\lambda\binom{X}{h}$ is {\it $r$-regular} if in every color class each element of $X$ occurs $r$ times. A one-regular color class is a {\it perfect matching}. We are interested in the necessary and sufficient conditions under which an $r$-regular coloring of $\lambda \binom{X}{h}$ can be embedded into an $s$-regular coloring of $\lambda \binom{Y}{h}$. 
Using algebraic techniques involving glueing together orbits of a suitably chosen cyclic group, the first author and Newman (Combinatorica 38 (2018), no. 6, 1309--1335) solved the case when $\lambda=1,r=s, \gcd (|X|,|Y|,h)=\gcd(|Y|,h)$. Using purely  combinatorial techniques, we nearly settle the case $h=4$. Two major challenges include finding all the necessary conditions, and obtaining the exact bound for $|Y|$.

It is  worth noting that completing partial symmetric latin squares is closely related to the case $\lambda =r=s=1, h=2$ which was solved by Cruse (J. Comb. Theory Ser. A 16 (1974), 18--22).

   \end{abstract}
   \subjclass[2010]{05C70, 05C65, 05C15}
   \keywords{embedding, factorization, edge-coloring, decomposition, Baranyai's theorem, amalgamation, detachment}
   \date{\today}

   \maketitle   
 
\section{Introduction} 
Consider the following $7 \times 18$ array containing   all the 126 4-subsets of $\{1,2,\dots, 9\}$. 

\begin{table}[h!]{\small \setlength\tabcolsep{2pt}
\begin{tabular}{cccccccccc|cccccccc}
\multicolumn{3}{l}{1235\hspace{1pt} 1246\hspace{1pt} 3456}\vline & 1268 & 1367 & 1457 & 2348 & 2378 & 4578 & 5678 & 1289 & 1489 & 1679 & 2349 & 2569 & 3459 & 3789 & 5679 \\
\multicolumn{3}{l}{1234\hspace{1pt} 1356\hspace{1pt} 2456}\vline & 1348 & 1357 & 1678 & 2467 & 2478 & 2568 & 3578 & 1469 & 1479 & 1569 & 2359 & 2689 & 2789 & 3479 & 3589 \\
\multicolumn{3}{l}{1245\hspace{1pt} 1346\hspace{1pt} 2356}\vline & 1368 & 1468 & 1478 & 2357 & 2567 & 2578 & 3478 & 1239 & 1259 & 1349 & 2489 & 3579 & 4679 & 5689 & 6789 \\
\multicolumn{3}{l}{1256\hspace{1pt} 1345\hspace{1pt} 2346}\vline & 1278 & 1347 & 1568 & 2358 & 2457 & 3678 & 4678 & 1459 & 1589 & 1789 & 2369 & 2479 & 2579 & 3469 & 3689 \\
\multicolumn{3}{l}{1236\hspace{1pt} 1456\hspace{1pt} 2345} \vline & 1247 & 1378 & 1578 & 2368 & 2678 & 3457 & 4568 & 1359 & 1379 & 1579 & 2389 & 2469 & 2679 & 4589 & 4689 \\
\cline{1-3}
1237 & 1248 & 1257 & 1467 & 1567 & 2367 & 2458 & 3458 & 3468 & 3568 & 1269 & 1369 & 1389 & 2379 & 2459 & 4569 & 4789 & 5789 \\
1238 & 1258 & 1267 & 1358 & 1458 & 2347 & 2468 & 3467 & 3567 & 4567 & 1249 & 1279 & 1689 & 2589 & 3489 & 3569 & 3679 & 4579 \\
\end{tabular}
}
\end{table}

In each row of this array, every element of $\{1,\dots, 9\}$ appears exactly eight times.  The $7 \times 10$ left subarray contains  all the 70 4-subsets of $\{1,\dots, 8\}$, and in each   row of this subarray  every element of $\{1,\dots, 8\}$ appears exactly five times. But there is more! The $5 \times 3$  top left subarray contains  all the 15 4-subsets of $\{1,\dots, 6\}$, and in each   row of this small subarray  every element of $\{1,\dots, 6\}$ appears exactly twice. Figure \ref{figcompl} provides a visual representation of this example, where colored 4-cycles correspond to  4-subsets. The goal of this paper is to construct such highly regular set systems.  First, we need to provide some background. 

Let $\cc G:= \lambda K_m^h$ be the collection of all $h$-subsets of a set $[m]:=\{1,\dots,m\}$ of {\it vertices} where each subset occurs $\lambda$ times. A {\it $q$-coloring} of $\cc G$ is a mapping $f:\cc G\rightarrow [q]$; elements of a color $j\in [q]$, written $\cc G(j)$,   form the {\it color class $j$} and the {\it order} of $\cc G(j)$ is $|\bigcup_{e\in \cc G(j)} e |$.  A coloring is {\it $r$-regular} if in each color class the {\it degree} (i.e. the number of occurrences) of each vertex  is $r$.  The problem of finding regular colorings of $\lambda K_m^h$ dates back to the  18th century. Two celebrated examples are the Sylvester's problem that asks for a one-regular coloring of $K_m^h$ in which each color class is of order $m$, and Steiner's problem that asks for an $\binom{r-1}{h-1}$-regular coloring of $K_m^h$ in which each color class is of order $r$. Sylvester's problem was settled in the 70s by Baranyai \cite{MR0416986}, and Steiner's problem was solved very recently for large $m$ by Keevash \cite{2014arXiv1401.3665K} and Glock et al. \cite{glock2016existence}. 

An {\it $r$-factorization} of $\lambda K_m^h$ is an $r$-regular coloring in which each color class is of order $m$.  In this paper, we are concerned with the following embedding analogue of regular colorings of $\lambda K_m^h$ which is closely related to Cameron's problem (see \cite[Question 1.2]{MR0419245}.  
\begin{question} \label{embrsprob1}
Find all values of $s$ and $n$ such that the  given $r$-factorization of $\cc G:=\lambda K_m^h$ can be extended to an $s$-factorization of $\lambda K_n^h$. 
\end{question}
Previously,  Problem \ref{embrsprob1} was solved for the following cases: $\lambda=r=s=1, h=2$ \cite{MR0329925} (this case is  closely related to completing partial symmetric latin squares), $\lambda=1, h=2$ \cite{MR1315436},  $\lambda=r=s=1$ \cite{MR1249714}, $\lambda=1, h=3$ \cite{MR3512664}, and $\lambda=1, r=s, \gcd (m,n,h)=\gcd(n,h)$ \cite{MR3910877}. The major obstacle for the case where $h\geq 3$ stems from the natural difficulty of generalizing graph theoretic results to hypergraphs. In this paper we nearly settle Problem \ref{embrsprob1} for the case when $h=4$.  
\begin{theorem} \label{embdrsh4compthmp1}
If $r\binom{n-1}{3}>s\binom{m-1}{3}$, then an $r$-factorization of $\lambda K_m^4$ can be extended to an $s$-factorization  of $\lambda K_n^4$ if and only if
\begin{align} \label{divcond+}
    &&
    4 \mid rm,
    &&
    4 \mid sn,
    &&
    r \mid \lambda\binom{m-1}{3},
    &&
    s \mid \lambda\binom{n-1}{3},
    &&
   1 \leq \dfrac{s}{r}\leq\binom{n-1}{3}/\binom{m-1}{3};
    &&
  \end{align}
\begin{align}\label{lboundn}
  n\geq \begin{cases} 
      2m & {\text{if}}\ s=r, \\
     \dfrac{4m}{3} & \text{if}\ s>r
   \end{cases};
\end{align}
\begin{align} \label{longnec0}
(n-m)\binom{m}{3}\geq \left(m-\dfrac{n}{2}\right)\left[\binom{n-1}{3}-\frac{s}{r}\binom{m-1}{3}\right];
\end{align}
\begin{align} \label{longnec1}
2(n-m)\binom{m}{3}+\binom{m}{2}\binom{n-m}{2}\geq \Big(m-\dfrac{n}{4}\Big)\left[\binom{n-1}{3}-\frac{s}{r}\binom{m-1}{3}\right].
\end{align}
\end{theorem}
\begin{theorem} \label{embdrsh4compthmp2}
If $n\geq 4m/3$ and $r\binom{n-1}{3}=s\binom{m-1}{3}$, then an $r$-factorization of $\lambda K_m^4$ can be extended to an $s$-factorization  of $\lambda K_n^4$ if and only if \eqref{divcond+} holds.
\end{theorem}

We note that none of the previous partial solutions to Problem \ref{embrsprob1} considered $\lambda>1$. Two major challenges in proving Theorems \ref{embdrsh4compthmp1} and \ref{embdrsh4compthmp2}  include finding all the necessary conditions (see Section \ref{neccsec}), and obtaining the exact bound for $n$. 
Judging from the literature  on embedding results for other combinatorial structures such as latin squares and Steiner triple systems, finding a sharp lower bound for $n$ in general is quite difficult. A noteworthy example is Lindner's conjecture \cite{MR0460213} that any partial Steiner triple system of order $m$ can be embedded in a Steiner triple system of order $n$ if $n \equiv 1,3 \Mod 6$ and $n \geq 2m + 1$, which despite numerous attempts, took over thirty years  to be resolved  \cite{MR2475426}. 

The layout of the paper is  as follows.  In Section \ref{neccsec}, we discuss the necessary conditions. In Section \ref{proofsketchsec} we provide the sketch of the proof of sufficiency.  Coloring the edges is divided into two parts, and the easier part is provided in  Section  \ref{col1sec}. Several crucial inequalities  will be needed before we are able to complete the coloring in Section \ref{col2sec}, and those are proven in Section \ref{ineqsec}. 

We end this section with some notation. A {\it hypergraph} $\mathcal G$ is a pair $(V(\mathcal G),E(\mathcal G))$ where $V(\mathcal G)$ is a finite set called the {\it vertex} set, $E(\mathcal G)$ is the {\it edge} multiset, where every edge is itself a multi-subset of $V(\mathcal G)$. This means that not only can an edge  occur multiple times in $E(\mathcal G)$, but also each vertex can have multiple occurrences within an edge. The total number of occurrences of a  vertex $v$ among all edges of $E(\mathcal G)$ is called the {\it degree}, $\dg_{\mathcal G}(v)$ of $v$ in $\mathcal G$. For two hypergraphs $\cc G$ and $\cc F$, $\cc G\backslash \cc F$ is the hypergraph whose vertex set is $V(\cc G)$ and whose edge set is $E(\cc G)\backslash E(\cc F)$. 
\begin{figure}[h!]
   \includegraphics[width=1\textwidth, angle=0]{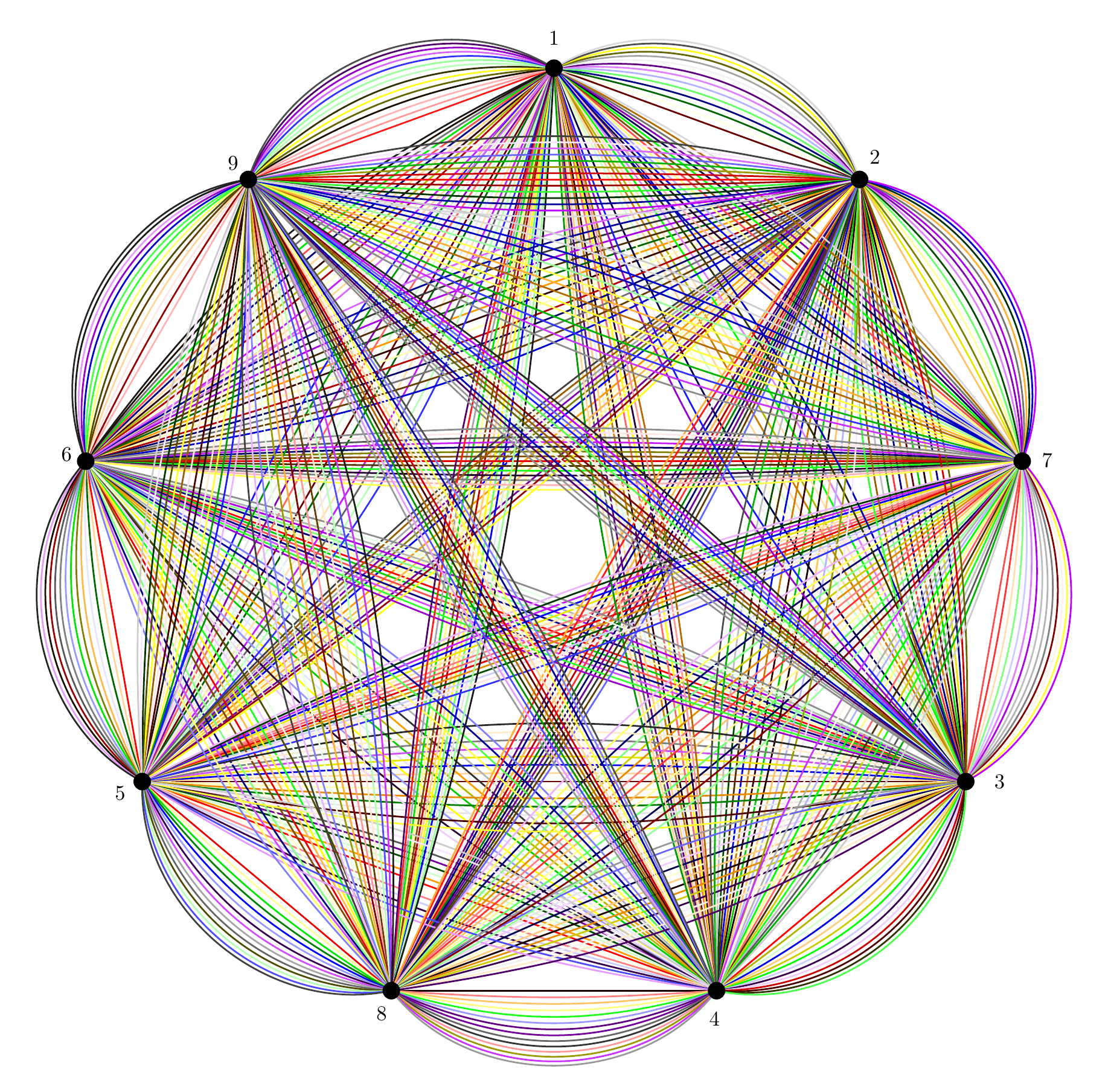}
  \caption{A 2-factorization of $K_6^4$ embedded into a 5-factorization of $K_8^4$ which is embedded into an 8-factorization of $K_9^4$}
\label{figcompl}
\end{figure}

\section{Necessary Conditions} \label{neccsec}
In order to avoid  trivial cases, we shall make the following assumptions: (i)  $n>m\geq 4$, (ii) if $m=4$, then $\lambda\geq 2$ and $r\geq 2$.

The following lemma will be used without further explanation when required.
\begin{lemma} \label{doublcount}
For $m,n\in \mathbb{N}$ with $n>m$, we have
\begin{enumerate}[label=\textup{({\alph*})}] 
\item  $\binom{n}{4}=\binom{m}{4}+(n-m)\binom{m}{3}+\binom{m}{2}\binom{n-m}{2}+m\binom{n-m}{3}+\binom{n-m}{4},$
\item $\binom{n-1}{3}=\binom{m-1}{3}+(n-m)\binom{m-1}{2}+(m-1)\binom{n-m}{2}+\binom{n-m}{3}$, and 
\item $m\left[\binom{n-1}{3}-\binom{m-1}{3}\right]=3(n-m)\binom{m}{3}+2\binom{m}{2}\binom{n-m}{2}+m\binom{n-m}{3}.$
\end{enumerate}
\end{lemma}
The proof is based on a simple double counting argument, and we shall skip it here. 

A triple $(m,r,\lambda)$ is {\it admissible} if $4 \mid rm$ and $r \mid \lambda\binom{m-1}{3}$. The following lemma settles the necessary conditions. 
\begin{lemma} If an $r$-factorization of $\lambda K_m^4$ can be extended to an $s$-factorization of $\lambda K_n^4$, then the following conditions hold.
\begin{enumerate}[label=\textup{(N{{\arabic*}})}] 
\item The triples $(m,r,\lambda)$ and $(n,s,\lambda)$ are admissible;
\item $1 \leq s/r\leq\binom{n-1}{3}/\binom{m-1}{3}$;
\item If $s=r$, then $n\geq 2m$;
\item $n\geq \frac{m}{3}(4-r/s)$;
\item If $1<s/r<\binom{n-1}{3}/\binom{m-1}{3}$, then $n\geq 4m/3$;
\item
$$(n-m)\binom{m}{3}\geq \left(m-\dfrac{n}{2}\right)\left[\binom{n-1}{3}-\frac{s}{r}\binom{m-1}{3}\right];$$
\item 
\begin{align*} 
2(n-m)\binom{m}{3}+\binom{m}{2}\binom{n-m}{2}\geq \left(m-\dfrac{n}{4}\right)\left[\binom{n-1}{3}-\frac{s}{r}\binom{m-1}{3}\right];
\end{align*}
\item 
If $s/r=\binom{n-1}{3}/\binom{m-1}{3}$, then
\begin{align*} 
  \frac{1}{s}\binom{n-1}{3}\leq \begin{cases} 
      \dbinom{m}{2}\dbinom{n-m}{2}+m\dbinom{n-m}{3} & {\text{if}}\ m(s-r)\equiv 1\Mod{3}, \\
      \dbinom{m}{2}\dbinom{n-m}{2}+\dfrac{m}{2}\dbinom{n-m}{3} & \text{if}\ m(s-r)\equiv 2\Mod{3}.
   \end{cases}
\end{align*}
\end{enumerate}
\end{lemma}
\begin{proof}
Suppose  that an $r$-factorization of $\lambda K_m^4$ is extended to an $s$-factorization of $\lambda K_n^4$. Since $\lambda K_m^4$ is $r$-factorable, $r$ divides the degree of each vertex in $\lambda K_m^4$. The existence of an $r$-factor in $\lambda K_m^4$ implies that $4 \mid rm$. Thus, $(m,r,\lambda)$ is admissible. A similar argument shows that  $(n,s,\lambda)$ is also admissible. Consequently,  $q:=\lambda \binom{m-1}{3}/r$ and $k:=\lambda \binom{n-1}{3}/s$ are integers. 

In order to extend an $r$-factorization of $\lambda K_m^4$ to an $s$-factorization of $\lambda K_n^4$, we must clearly have $s\geq r$. Moreover, the number of colors used in an $r$-factorization of $\lambda K_m^4$, $q$, is no more than the number of colors used in an $s$-factorization of $\lambda K_n^4$, $k$, and so (N2) holds.

For the rest of the proof, we shall refer to the $m$ vertices of $\lambda K_m^4\subseteq \lambda K_n^4$, and the remaining $n-m$ vertices of $\lambda K_n^4\backslash \lambda K_m^4$ as the old vertices, and the new vertices, respectively. Moreover, $\kappa:=\kappa_1\cup \kappa_2$ where $\kappa_1:=\{1,\dots,q\}$ is the set of the old colors (those used in the coloring of $\lambda K_m^4$), and $\kappa_2:=\{q+1,\dots,k\}$ is the set of the new colors (those used only in the coloring of $\lambda K_n^4\backslash \lambda K_m^4$). For $j\in \kappa$,  let $a_j,e_j,f_j, g_j$, and $\ell_j$ be the number of edges colored $j$ in $\lambda K_n^4$ that are incident with exactly $0,1,2,3$, and 4 new vertices, respectively.

To prove (N3) suppose that $s=r$, and let  $j\in \kappa_1$. We cannot have any edges colored $j$ between the old vertices and the new vertices. Therefore, to form an $r$-factor, we can only use the new edges, and so $n-m\geq 4$. In order to form an $r$-factor in $\lambda K_n^4$ with color $j$, we need to have  $r(n-m)/4$ edges colored $j$ between the new vertices. Hence, 
$$\lambda \binom{n-m}{4}\geq q\frac{r(n-m)}{4}.$$
This implies $\frac{4}{n-m}\binom{n-m}{4}\geq \binom{m-1}{3}$. Therefore, $\binom{n-m-1}{3}\geq\binom{m-1}{3}$, or equivalently, $n-m-1\geq m-1$, and so (N3) is satisfied.

In an $s$-factorization of $\lambda K_n^h$, each of the  new vertices   is adjacent with exactly $s$ edges of each color, so all the  new vertices are adjacent with at most $s(n-m)$ edges of each color. We have $s(n-m)+a_j\geq sn/4$ for $j\in \kappa$. Since $a_j=rm/4$ for $j\in \kappa_1$, we have $s(n-m)+rm/4\geq sn/4$, which proves (N4). Moreover, if $k>q$, then $\kappa_2\neq\varnothing$ and since $a_j=0$ for $j\in \kappa_2$,  we have $s(n-m)\geq sn/4$ which  proves (N5).

To prove (N6) and (N7), observe that for $k=q$, the right hand side of (N6)  and (N7) is zero, and  there is nothing to prove. So let us assume that $k>q$. Within each new color class of $\lambda K_n^4$, the degree sum of all the old vertices is $sm$.  Therefore, for any $j\in\kappa_2$,
$$sm=3e_j+2f_j+g_j.$$
Since the number of edges in each new color class of $\lambda K_n^4$ is $sn/4$, we have
$$sn/4=e_j+f_j+g_j+\ell_j.$$
Therefore, 
$$sm-\frac{sn}{4}=2e_j+f_j-\ell_j\leq 2e_j+f_j.$$
By taking the sum over all new colors, we have
$$(k-q)(sm-\frac{sn}{4})\leq 2\lambda(n-m)\binom{m}{3}+\lambda\binom{m}{2}\binom{n-m}{2},$$
which proves (N7). Similarly, we have 
$$sm-\frac{sn}{2}=e_j-g_j-2\ell_j\leq e_j,$$
and so
$$(k-q)(sm-\frac{sn}{2})\leq \lambda(n-m)\binom{m}{3},$$
which proves (N6).

To prove (N8), first we show that
\begin{equation} \label{k=qmsrnz}
\mbox{ If }k=q, \mbox{ and }m(s-r)\nequiv 0\Mod{3}, \mbox{ then }n\geq m+2. 
\end{equation}
Suppose by the contrary that $n=m+1$. 
Since $k=q$, we have $\binom{m}{3}/s=\binom{m-1}{3}/r$, and so $m(s-r)=3s$, which is a contradiction. Within the new edges of $\lambda K_n^4$, the degree sum of all the old vertices in each old color class  is $m(s-r)$.  Therefore, 
$$m(s-r)=3e_j+2f_j+g_j\equiv 2f_j+g_j \Mod 3\mbox { for } j\in \kappa_1.$$
There are two cases to consider. 
\begin{enumerate} [label=({\alph*})]
\item If $m(s-r)\equiv 1\Mod{3}$, then $2f_j+g_j\equiv 1\Mod {3}$, and so $f_j+g_j\geq 1$ for each $j\in \kappa_1$. 
Hence, 
$$\lambda{m \choose2}{n-m\choose2}+\lambda m{n-m\choose3}\geq q.$$
\item If $m(s-r)\equiv 2\Mod{3}$, then $2f_j+g_j\equiv 2\Mod{3}$, and so $2f_j+g_j\geq 2$ for each $j\in \kappa_1$. Hence, 
$$2\lambda {m\choose2}{n-m\choose2}+\lambda m{n-m\choose3}\geq 2q.$$
\end{enumerate}
\end{proof}
Although, we will need  the necessary conditions (N4) and (N8),  in the next two lemmas, we show that (N4) and (N8) are redundant. 
\begin{lemma}\label{A7}
(N5) implies (N4).
\end{lemma}
\begin{proof}
If $k>q$, then by (N5), $n\geq 4m/3$ and so (N4) clearly holds. Let us assume that $k=q$. We need to show that $n\geq \frac{m}{3}\left(4-\left[\binom{m-1}{3}/\binom{n-1}{3}\right]\right)$, or equivalently, 
$$3n\binom{n-1}{3}-4m\binom{n-1}{3}+m\binom{m-1}{3}\geq0.$$
Since $n\geq m+1$ and $m\geq 4$, we have
\begin{align*}
&9n\binom{n-1}{3}-12m\binom{n-1}{3}+3m\binom{m-1}{3}\\
=&36\binom{n}{4}-12m\binom{n-1}{3}+12\binom{m}{4}\\
=&12\left[\binom{m}{2}\binom{n-m}{2}+2m\binom{n-m}{3}+3\binom{n-m}{4}\right]\\
=&\binom{n-m}{2}\left( n(3n+2m-15) + m^2-7m+18\right)\\
\geq & \binom{n-m}{2}\left(3(m+1)+2m-15\right)\geq 0.
\end{align*}
\end{proof}

\begin{lemma}
If $m,n,r,s\in \mathbb{N}$ such that $n>m, s>r$, and $s/r=\binom{n-1}{3}/\binom{m-1}{3}$, then
\begin{align*} 
  \frac{1}{s}\binom{n-1}{3}\leq \begin{cases} 
      \dbinom{m}{2}\dbinom{n-m}{2}+m\dbinom{n-m}{3} & {\text{if}}\ m(s-r)\equiv 1\Mod{3}, \\
      \dbinom{m}{2}\dbinom{n-m}{2}+\dfrac{m}{2}\dbinom{n-m}{3} & \text{if}\ m(s-r)\equiv 2\Mod{3}.
   \end{cases}
\end{align*}
\end{lemma}

\begin{proof}
By \eqref{k=qmsrnz}, $n\geq m+2$. We show that
\begin{equation} \label{k=qnm+2sr}
\mbox{ If }k=q, \mbox{ and }n=m+2, \mbox{ then }s\geq r+2. 
\end{equation}
Suppose on the contrary that $k=q, n=m+2$, but $s=r+1$. Since $s/r=\binom{n-1}{3}/\binom{m-1}{3}$, we have $s(m-2)(m-3)=rm(m+1)$, or equivalently, $m^2-(6r+5)m+6(r+1)=0$. This implies that $m=\left(6r+5\pm\sqrt{36r^2+36r+1}\right)/2$. Since $(6r+2)^2< 36r^2+36r+1<(6r+3)^2$, $36r^2+36r+1$ is not a perfect square, but $m$ is an integer, this is a contradiction.

Since $s/r=\binom{n-1}{3}/\binom{m-1}{3}$, we have
\begin{align*}
\frac{s}{r}\binom{m-1}{3}=&\binom{n-1}{3}\\
=&\binom{m-1}{3}+(n-m)\binom{m-1}{2}+(m-1)\binom{n-m}{2}+\binom{n-m}{3}.
\end{align*}
Therefore, 
$$\binom{m-1}{3}=\frac{r}{s-r}\left[(n-m)\binom{m-1}{2}+(m-1)\binom{n-m}{2}+\binom{n-m}{3}\right].$$
To complete the proof, there are two cases to consider.
\begin{enumerate}[label=({\roman*})] 
\item If $m(s-r)\equiv 1\Mod{3}$,  we need to show that 
\begin{align*}
rm\binom{n-m}{3}+&r\binom{m}{2}\binom{n-m}{2} \\ 
\  &\geq\frac{r}{s-r}\left[(n-m)\binom{m-1}{2}+(m-1)\binom{n-m}{2}+\binom{n-m}{3}\right],
\end{align*}
or, equivalently $$[m(s-r)-1]\binom{n-m}{3}+\left[m(s-r)-2\right]\frac{m-1}{2}\binom{n-m}{2}\geq (n-m)\binom{m-1}{2}.$$
If $n-m=2$, then $s-r\geq 2$, and so it is enough to show that $(2m-2)\frac{m-1}{2}\geq 2\binom{m-1}{2}$, which is clearly true. If $n-m\geq3$, since $s-r\geq1$, it is enough to show that $(m-1)\binom{n-m}{3}+\binom{m-1}{2}\binom{n-m}{2}\geq (n-m)\binom{m-1}{2}$, or equivalently, 
$$\binom{n-m}{2}\left[(m-1)(n-m-2)+3\binom{m-1}{2}\right]\geq 3(n-m)\binom{m-1}{2}.$$
Since $n-m\geq3$, we have $\binom{n-m}{2}\geq n-m$. Therefore, it is enough to show that $$(m-1)(n-m-2)+3\binom{m-1}{2}\geq 3\binom{m-1}{2},$$
which clearly holds.

\item If $m(s-r)\equiv 2\Mod{3}$,  we need to show that 
$$
[m(s-r)-2]\binom{n-m}{2}\left(\frac{n+2m-5}{3}\right)\geq 2(n-m)\binom{m-1}{2}.
$$
If $n-m=2$, then $s-r\geq 2$, and so it is enough to show that $(2m-2)\frac{3m-3}{3}\geq 4\binom{m-1}{2}$, which is clearly true.
If $n-m\geq 3$, then we have $\binom{n-m}{2}\geq n-m$. Therefore, since $s-r\geq 1$, it suffices to show that $(n+2m-5)/3\geq m-1$. This is equivalent to $n-m\geq 2$, which is trivial.
\end{enumerate}
\end{proof}

\begin{remark}\textup{
Neither  Condition \eqref{longnec0} nor Condition \eqref{longnec1}  can  be eliminated from Theorem \ref{embdrsh4compthmp1}. 
 For example, a $4$-factorization of $K_7^4$ cannot be extended to a 6-factorization of $K_{10}^4$, and a $6$-factorization of $K_{10}^4$ cannot be extended to a 7-factorization of $K_{16}^4$. In the first example all conditions but  \eqref{longnec0} hold, and in the second example all conditions but  \eqref{longnec1} hold.
}\end{remark}

\section{Sketch of Proof} \label{proofsketchsec}
Throughout the rest of this paper, we shall assume that the conditions \eqref{divcond+}--\eqref{longnec1} (or equivalently, (N1)--(N8)) hold, $n\geq 4m/3$, and that 
\begin{align*} 
    &&
    q:=\dfrac{\lambda}{r}\binom{m-1}{3}, 
    &&
    k:=\dfrac{\lambda}{s}\binom{n-1}{3}.
    &&
  \end{align*}
Since $(m,r,\lambda)$ and $(n,s, \lambda)$ are admissible, both $q$ and $k$ are integers. Let
\begin{align*} 
    &&
    \kappa_1=\{1,\dots,q\}, 
    &&
    \kappa_2=\{q+1,\dots,k\},
    &&
    \kappa=\kappa_1\cup\kappa_2.
	&&
\end{align*}
Let $\cc F$ be a 2-vertex hypergraph with $V(\cc F)=\{u,v\}$. In order to describe the edge set of $\cc F$, first we need to introduce some notation. The {\it multiplicity} of an edge $e$ in $\cc F$, written $\mult(e)$, is the number of repetitions of $e$ in $\cc F$. A {\it $u^{i}v^{j}$-edge} is an edge in which vertex $u$ occurs $i$ times and vertex $v$ occurs $j$ times. When we color the edges of  $\cc F$,  we use $\dg_j(v)$ and $\mult_j(e)$ for the degree of $v$, and the multiplicity of $e$ in color class $j$. The following describes the edge set of $\cc F$.
 \begin{align}\label{def2vhyp}
    &&
    \mult(u^iv^{4-i})=\lambda\binom{m}{i}\binom{n-m}{4-i} \quad \mbox{ for } 0\leq i\leq 3.
    &&
  \end{align}
Observe that $\cc F$ can be obtained by identifying all the  $m$ old  vertices of $\cc G:=\lambda K_n^4\backslash \lambda K_m^4$ by a vertex $u$, and identifying all the remaining $n-m$ new vertices with $v$. We say that $\cc F$ is an {\it amalgamation} of $\cc G$, and that $\cc G$ is a {\it detachment} of $\cc F$.  
  
We think of the given $r$-factorization of $\lambda K_m^4$ as a $q$-coloring of $\lambda K_m^4$ in which each color class induces an $r$-factor. In order to extend the $r$-factorization of $\lambda K_m^4$ to an $s$-factorization of $\lambda K_n^4$, we need to color  $\cc G$ with $k$ colors such that each color  class of $\lambda K_n^4$ induces an $s$-factor. If  we can obtain such a coloring, then in the amalgamation $\cc F$ of $\cc G$, $\dg_j(u)= m(s-r)$  for $j\in \kappa_1$, $\dg_j(u)= sm$  for $j\in \kappa_2$, and $\dg_j(v)= s(n-m)$  for $j\in \kappa$.  More importantly, by the following lemma which is an immediate consequence of a result of the first author (see \cite[Theorem 4.1]{MR2942724}),  the converse of the previous statement is also true.   
\begin{lemma}\label{f}
If  the hypergraph $\cc F$ described in \eqref{def2vhyp} can be colored so that 
\begin{equation}\label{degj}
  \dg_j(x)=\begin{cases} 
      m(s-r) & {\text{if}}\ x=u,j\in \kappa_1, \\
      sm & \text{if}\ x=u,j\in \kappa_2, \\
      s(n-m) & \text{if}\ x=v, j\in \kappa,
   \end{cases}
\end{equation}
then an $r$-factorization of $\lambda K_m^4$ can be extended to an $s$-factorization of $\lambda K_n^4$.
\end{lemma}
Thus, the problem of  extending an $r$-factorization of $\lambda K_m^4$ to an $s$-factorization of $\lambda K_n^4$ is reduced to coloring of the amalgamation $\cc F$ of $\cc G$. Although it is much easier to color $\cc F$ than to color $\cc G$, it is particularly very difficult to color $\cc F$ when $n$ and $m$ are very close to each other. The rest of the paper is devoted to a coloring of $\cc F$ ensuring that \eqref{degj} is satisfied. As we shall see in the next section, provided  we  carefully  color the $u^3v$-edges  and $u^2v^2$-edges,  coloring the remaining edges of $\cc F$ is straightforward (see Lemma \ref{m1m2lemma}). Before we can color the  $u^3v$-edges  and $u^2v^2$-edges, we need to prove several crucial inequalities. We shall do this in Section \ref{ineqsec} following  the  coloring of the  $u^3v$-edges  and $u^2v^2$-edges in Section \ref{col2sec}.

\section{Colorings I} \label{col1sec}
In this section we will show that if we can color  the  $u^3v$-edges  and $u^2v^2$-edges so that certain conditions are met (see Condition \eqref{M1equ3v}), then it is easy to color the $uv^3$-edges  and $v^4$-edges.  To achieve this goal, let us introduce the following fixed parameters. 
\begin{align*}
\hspace{1cm}
&
 \left \{
  \begin{aligned}
    &\iota_1:=sm-\dfrac{sn}{2}-\dfrac{rm}{2}, \\
&\iota_2:=sm-\dfrac{sn}{2}, &\mbox{ if } \kappa_2\neq \varnothing,\\
   &  \rho_1:=\dfrac{sm}{3}-\dfrac{rm}{3},\\
  &    \rho_2:=\dfrac{sm}{3},&\mbox{ if } \kappa_2\neq \varnothing,\\
   & \rho'_1:=\frac{sm}{2}-\frac{sn}{8}-\frac{3rm}{8},\\
 & \rho'_2:=\frac{sm}{2}-\frac{sn}{8},&\mbox{ if } \kappa_2\neq \varnothing.\\
\end{aligned} \right.
\end{align*}

Since by (N1), $(m,r,\lambda)$ and $(n,s,\lambda)$ are admissible, $\iota_1$ and $\iota_2$ are integers. By (N2), $s\geq r\geq 1$,  and so we have that $\rho_1\geq 0$ and $\rho_2> 0$.  Observe that $\iota_1$ and  $\iota_2$ are not necessarily non-negative. In fact,  $\iota_1\geq0 $ if an only if  $n\leq(2-\frac{r}{s})m$, and   $\iota_2\geq0$ if and only if $n\leq2m$. In addition, $\rho_1, \rho_2,\rho'_1, \rho'_2$ are not necessarily integers.  It is also easy to see that $\rho'_1\geq0$ if and only if $n\leq(4-\frac{3r}{s})m$, and $\rho'_2\geq0$ if and only if $n\leq4m$. By (N4), $\iota_1\leq\rho'_1\leq \rho_1$. Finally,  by (N5) for $k>q$, $n\geq 4m/3$, so we have that $\iota_2\leq\rho'_2\leq \rho_2$ (for  $k>q$).

Once we color the $u^3v$-edges, we can introduce the following further parameters. 
\begin{align*}
\hspace{1cm}
&
 \left \{
  \begin{aligned}
  &  \iota_{1j}:= sm-\frac{sn}{4}-2\mult_j(u^3v)-\frac{3rm}{4}, &j\in \kappa_1\\
   &   \iota_{2j}:=sm-\frac{sn}{4}-2\mult_j(u^3v),&j\in \kappa_2\\
     & \rho_{1j}:=\frac{sm}{2}-\frac{3}{2}\mult_j(u^3v)-\frac{rm}{2},&j\in \kappa_1\\
      & \rho_{2j}:=\frac{sm}{2}-\frac{3}{2}\mult_j(u^3v),&j\in \kappa_2.
  \end{aligned} \right.
\end{align*}

Since  $(m,r,\lambda)$ and $(n,s,\lambda)$ are admissible, $\iota_{1j}\in \mathbb{Z}$ for $j\in \kappa_1$, and $\iota_{2j}\in \mathbb{Z}$ for $j\in \kappa_2$. Observe that $\rho_{1j}$ (for $j\in \kappa_1$), and $\rho_{2j}$ (for $j\in \kappa_2$) are not necessarily integers. Moreover, 
\begin{enumerate} [label=({\roman*})] 
\item for $j\in \kappa_1$,  $\rho_{1j}\geq0$  if and only if $\mult_j(u^3v)\leq \rho_1$,
\item for $j\in \kappa_2$,  $\rho_{2j}\geq0$  if and only if $\mult_j(u^3v)\leq \rho_2$,
\item for $j\in \kappa_1$, $\iota_{1j}\geq 0$   if and only if $\mult_j(u^3v)\leq \rho_1'$,  
\item for $j\in \kappa_2$, $\iota_{2j}\geq 0$  if and only if $\mult_j(u^3v)\leq \rho_2'$,
\item for $j\in \kappa_1$, $\rho_{1j}\geq \iota_{1j}$   if and only if $\mult_j(u^3v)\geq\iota_1$, and 
\item for $j\in \kappa_2$, $\rho_{2j}\geq \iota_{2j}$  if and only if $\mult_j(u^3v)\geq\iota_2$.
\end{enumerate}

We have summarized all the useful information about $\iota_1, \iota_2, \rho_1, \rho_2, \rho_1', \rho_2', \iota_{1j}, \rho_{1j}$ for $j\in \kappa_1$, and $\iota_{2j},\rho_{2j}$ for $j\in \kappa_2$   in Table \ref{formulatable}.
	
\begin{lemma} \label{m1m2lemma}
Suppose that the  $u^3v$-edges and $u^2v^2$-edges of $\cc F$ can be colored such that
\begin{equation}\label{M1equ3v}
\left\{\!\begin{aligned}
 \iota_i&\leq\mult_j(u^3v)\ \leq \rho_i \\
 \iota_{ij}&\leq\mult_j(u^2v^2)\leq \rho_{ij}
\end{aligned}\right\}
\quad j\in \kappa_i, i=1,2.
\end{equation} 
Then the remaining edges of $\cc F$ can be colored so that \eqref{degj} holds. Accordingly, if \eqref{M1equ3v} is satisfied, then  an $r$-factorization of $\lambda K_m^4$ can be extended to an $s$-factorization of $\lambda K_n^4$.
\end{lemma}
\begin{proof} 
Suppose that the  $u^3v$-edges and $u^2v^2$-edges of $\cc F$ can be colored such that \eqref{M1equ3v} holds. First, we claim that we can color the $uv^3$-edges such that
\begin{equation} \label{uv3eq}
   \mult_j(uv^3) = 2\rho_{ij}-2\mult_j(u^2v^2)\quad j\in \kappa_i, i=1,2.
  \end{equation}
By \eqref{M1equ3v}, $\mult_j(u^2v^2)\leq \rho_{1j}$ for $j\in \kappa_1$, and therefore, $\mult_j(uv^3)\geq 0$ for $j\in \kappa_1$. Likewise, $\mult_j(u^2v^2)\leq \rho_{2j}$ for $j\in \kappa_2$, and so, $\mult_j(uv^3)\geq 0$ for $j\in \kappa_2$. Moreover, $2\rho_{ij}\in \mathbb{Z}$ for $j\in \kappa_i, i=1,2$, and so  $\mult_j(uv^3) \in \mathbb{Z}$ for $j\in \kappa$. 
  Hence, the following confirms that the coloring of the $uv^3$-edges satisfying \eqref{uv3eq} is possible.
\begin{align*}
\sum\nolimits_{j\in\kappa}\mult_j(uv^3)&= \sum\nolimits_{j\in\kappa_1}\mult_j(uv^3)+\sum\nolimits_{j\in\kappa_2}\mult_j(uv^3)\\
    &=qm(s-r)-3\sum\nolimits_{j\in\kappa_1}\mult_j(u^3v)-2\sum\nolimits_{j\in\kappa_1}\mult_j(u^2v^2)\\
& \ \ \ +sm(k-q)-3\sum\nolimits_{j\in\kappa_2}\mult_j(u^3v)-2\sum\nolimits_{j\in\kappa_2}\mult_j(u^2v^2)\\
&=ksm-qrm-3\lambda(n-m)\binom{m}{3}-2\lambda\binom{m}{2}\binom{n-m}{2}\\
&=\lambda m\left[\binom{n-1}{3}-\binom{m-1}{3}\right]-3\lambda(n-m)\binom{m}{3}-2\lambda\binom{m}{2}\binom{n-m}{2}\\
&=\lambda m\binom{n-m}{3}=\mult(uv^3).
\end{align*}

Now, we claim that we can color the $v^4$-edges such that
\begin{equation} \label{v4eq}
   \mult_j(v^4) = \mult_j(u^2v^2)- \iota_{ij}\quad j\in \kappa_i, i=1,2.
  \end{equation}
By \eqref{M1equ3v}, $\mult_j(u^2v^2)\geq \iota_{1j}$ for $j\in \kappa_1$, and  $\mult_j(u^2v^2)\geq \iota_{2j}$ for $j\in \kappa_2$. Therefore, $\mult_j(v^4)\geq 0$ for $j\in \kappa$. Moreover, $\iota_{ij}\in \mathbb{Z}$ for $j\in \kappa_i, i=1,2$, and so  $ \mult_j(v^4)  \in \mathbb{Z}$ for $j\in \kappa$.  The following confirms that the coloring of the $v^4$-edges satisfying \eqref{v4eq} is possible.
\begin{align*}
\sum\nolimits_{j\in\kappa}\mult_j(v^4)&= \sum\nolimits_{j\in\kappa_1}\mult_j(v^4)+\sum\nolimits_{j\in\kappa_2}\mult_j(v^4)\\
    &=q(\frac{sn}{4}-sm+\frac{3rm}{4})+(k-q)(\frac{sn}{4}-sm)\\
&\qquad\qquad +\ 2\sum\nolimits_{j\in\kappa}\mult_j(u^3v)+\sum\nolimits_{j\in\kappa}\mult_j(u^2v^2)\\
&=\frac{ksn}{4}-ksm+\frac{3qrm}{4}+2\lambda(n-m)\binom{m}{3}+\lambda\binom{m}{2}\binom{n-m}{2}\\
&=\lambda\binom{n}{4}-\lambda m\binom{n-1}{3}+3\lambda\binom{m}{4}+2\lambda(n-m)\binom{m}{3}+\lambda\binom{m}{2}\binom{n-m}{2}\\
&=\lambda\binom{n}{4}-\lambda m\binom{n-1}{3}+3\lambda\binom{m}{4}+\lambda m\left[\binom{n-1}{3}-\binom{m-1}{3}\right]-\lambda\binom{n}{4}\\
&+\ \lambda\binom{m}{4}+\lambda\binom{n-m}{4}\\
&=\lambda\binom{n-m}{4}=\mult(v^4).
\end{align*}

Finally, we prove that our coloring satisfies \eqref{degj}. For $j\in\kappa$, 
\begin{align*}
    \dg_j(x) &=
     \begin{cases} 
     3\mult_j(u^3v)+2\mult_j(u^2v^2)+\mult_j(uv^3)
       & \mbox { if }x=u \\
     \mult_j(u^3v)+2\mult_j(u^2v^2)+3\mult_j(uv^3)+4\mult_j(v^4) & \mbox { if }x=v \end{cases}\\
     &=
     \begin{cases} 
     3\mult_j(u^3v)+2\rho_{ij}        &  \mbox { if }x=u, j\in \kappa_i, i=1,2 \\
     \mult_j(u^3v)+
     6\rho_{ij}-4\iota_{ij} & \mbox { if }x=v, j\in \kappa_i, i=1,2 
     \end{cases}\\
 &=
     \begin{cases}      
     \begin{cases} 
     m(s-r) & \mbox { for }j\in\kappa_1 \\
      sm & \mbox { for }j\in\kappa_2 \end{cases}
       & \mbox { if }x=u \\
     s(n-m) & \mbox { if }x=v. \end{cases}  \end{align*}
\end{proof}

\section{Colorings II} \label{col2sec}
In this section we  show that we can color the   $u^3v$-edges  and $u^2v^2$-edges of $\cc F$   such that \eqref{M1equ3v} holds. Since we will frequently deal with $u^3v$-edges  and $u^2v^2$-edges, let us make the following abbreviation.

\begin{align*}
&
 \left \{
  \begin{aligned}
    &e:=\lambda(n-m)\binom{m}{3}, \\
&f:=\lambda\binom{m}{2}\binom{n-m}{2},\\
  & e_j:=\mult_j(u^3v), &\mbox{ for }  j\in \kappa,\\
 & f_j:=\mult_j(u^2v^2),&\mbox{ for }  j\in \kappa.
\end{aligned} \right.
\end{align*}
 The following simple lemma will be quite useful.
\begin{lemma} \label{averaginglemma}
Let $c, k\in \mathbb{N}$ and $S=\{1,\dots, k\}$. Let $a_i\in \mathbb{Z}, b_i\in \mathbb{R}$ with $b_i\geq 0,  a_i\leq b_i$ for $i\in S$.  Moreover, suppose that $a_i\geq 0$ for $i\in I\subseteq S$, and $a_i< 0$ for $i\in S\backslash I$. Consider the following system.
\begin{equation} \label{syseq}
 \left \{
  \begin{aligned}
     &\sum\nolimits_{i\in \kappa} x_i=c,&\\
  &\ a_i\leq x_i\leq b_i, &i\in S,\\
     &\ x_i\in \mathbb{N}\cup \{ 0\}, & i\in S.
\end{aligned} \right.
\end{equation}
Then \eqref{syseq} has  a solution if and only if $$\sum\nolimits_{i\in I}  a_i\leq c  \leq \sum\nolimits_{i\in S} \lfloor b_i\rfloor.$$ 
\end{lemma}
\begin{proof}
First, suppose that \eqref{syseq} has a solution $(x^*_1,\dots,x^*_k)$. Since $a_i\leq x^*_i\leq b_i$ for $i\in S$ and $x_i$s are non-negative integers, we have $a_i\leq x^*_i\leq \lfloor b_i \rfloor $ for $i\in I$, and $0\leq x^*_i\leq \lfloor b_i \rfloor $ for $i\in S\backslash I$. 
Combining this with $\sum\nolimits_{i\in S} x^*_i=c$ we have 
$$\sum\nolimits_{i\in I} a_i\leq c\leq \sum\nolimits_{i\in S} \lfloor b_i \rfloor.$$
Conversely, suppose that $\sum\nolimits_{i\in I}  a_i\leq c  \leq \sum\nolimits_{i\in S} \lfloor b_i\rfloor$. Let $x_i^*=a_i$ for $i\in I$, and $x_i^*=0$ for $i\in S\backslash I$. Clearly, (i) $x_i^*\in \mathbb{N}\cup \{ 0\}$ for $ i\in S$. 
Since by our hypothesis, $a_i\leq b_i$ for $i\in S$,  we have that (ii) $a_i\leq x_i^*\leq b_i$ for $i\in S$. At this point, $\sum\nolimits_{i\in S} x_i^*=\sum\nolimits_{i\in I} a_i\leq c$, but since $c  \leq \sum\nolimits_{i\in S} \lfloor b_i\rfloor$, we can increase the value of $x_i^*$s without violating conditions (i) and (ii) so that their sum is exactly $c$. 
\end{proof}

By Lemma \ref{m1m2lemma}, if we show that  the following system has a solution, then we are done.
\begin{align} \label{mastersys}
 &\left \{
  \begin{aligned}
     &\sum\nolimits_{j\in \kappa} e_j=e,&\\
     &\sum\nolimits_{j\in \kappa} f_j=f,&\\
  &\ \iota_1\hspace{3pt}\leq e_j\leq \rho_1, &j\in \kappa_1,\\
  &\ \iota_2\hspace{3pt}\leq e_j\leq \rho_2, &j\in \kappa_2,\\
  &\ \iota_{1j}\leq f_j\leq \rho_{1j}, &j\in \kappa_1,\\
  &\ \iota_{2j}\leq f_j\leq \rho_{2j}, &j\in \kappa_2.
\end{aligned} \right.
\end{align} 
By Lemma \ref{averaginglemma}, this system has a solution if and only if 
\begin{alignat}{5}
\sum_{j\in \kappa_1, \iota_1\geq 0}  \iota_{1}&+\sum_{j\in \kappa_2, \iota_2\geq 0}  \iota_{2} & \leq  e &\leq&\hspace{-.5cm} q \lfloor \rho_1\rfloor\ \ +&(k-q) \lfloor \rho_2\rfloor,\label{firsteineq}\\
\sum_{j\in J_1}  \iota_{1j}&+\quad \sum_{j\in J_2}  \iota_{2j}  &\leq  f & \leq & \ \sum_{j\in \kappa_1} \lfloor \rho_{1j}\rfloor +&\sum_{j\in \kappa_2} \lfloor \rho_{2j}\rfloor, \label{secondfineq}
\end{alignat}

where for $i=1,2$, 
$$
J_i:=\{j\in \kappa_i \ |\ \iota_{ij}\geq 0\}= \{j\in \kappa_i \ |\ e_j\leq \rho'_i\}.
$$
Since $\iota_2>\iota_1$, either $\iota_2>\iota_1\geq 0$, or $\iota_2\geq 0\geq \iota_1$, or $0\geq \iota_2>\iota_1$. Thus, 
\begin{align*}
\sum\nolimits_{j\in \kappa_1, \iota_1\geq 0}  \iota_{1}+\sum\nolimits_{j\in \kappa_2, \iota_2\geq 0}  \iota_{2}= \begin{cases} 
      (k-q)\iota_2+q\iota_1 & {\text{if}}\ \iota_2>\iota_1\geq 0, \\
      (k-q)\iota_2 & {\text{if}}\ \iota_2\geq 0\geq \iota_1, \\
      0  & \text{if}\ 0\geq \iota_2>\iota_1.
   \end{cases}
\end{align*}
But by Lemma \ref{iota12leq} and (N6)
\begin{align*}
e\geq  \begin{cases} 
      (k-q)\iota_2+q\iota_1 & {\text{if}}\ \iota_2>\iota_1\geq 0, \\
      (k-q)\iota_2 & {\text{if}}\ \iota_2\geq 0\geq \iota_1, \\
      0  & \text{if}\ 0\geq \iota_2>\iota_1.
   \end{cases}
\end{align*}
Again, by Lemma \ref{iota12leq},
$$
e \leq q \lfloor \rho_1\rfloor+(k-q) \lfloor \rho_2\rfloor. 
$$
So that settles \eqref{firsteineq}. Unfortunately,  \eqref{secondfineq} depends on \eqref{firsteineq}, so we cannot necessarily go on  to show that \eqref{secondfineq} is satisfied. The following example  illustrates this issue. 
\begin{example}\textup{
 Let $m=6, r= 2, n=9, s=4, \lambda=1$. It is easy to verify that   conditions \eqref{divcond+}--\eqref{longnec1} of Theorem \ref{embdrsh4compthmp1} are satisfied, and that $q=5, k=14, e=60, f=45, \iota_1=0, \rho_1=4, \iota_2=6, \rho_2=8$. Clearly, $e_1=4, e_2=e_3=e_4=e_5=0, e_6=8, e_7=e_8=e_9=e_{10}=e_{11}=e_{12}=e_{13}=e_{14}=6$ is a solution to 
\begin{align*}
 &\left \{
  \begin{aligned}
     &\sum\nolimits_{j=1}^{14} e_j=60,&\\
  &\ 0\leq e_j\leq 4, &j&= 1,\dots,5,\\
  &\ 6\leq e_j\leq 8, &j&=6,\dots,14.
\end{aligned} \right.
\end{align*}  
 For $j=1,\dots,5$, $\iota_{1j}=6-2e_j, \rho_{1j}=6-3e_j/2$, and for  $j=6,\dots,14$, $\iota_{2j}=15-2e_j, \rho_{2j}=12-3e_j/2$. Therefore, $\iota_{11}=-2, \rho_{11}=0$, and $\iota_{1j}=\rho_{1j}=6$ for $j=2,\dots,5$. Moreover, $\iota_{26}=-1, \rho_{26}=0$, and $\iota_{2j}= \rho_{2j}=3$ for $j=7,\dots,14$. 
 The following system has no solution.
 \begin{align*}
 &\left \{
  \begin{aligned}
     &&\hspace{-.4cm}\sum\nolimits_{j=1}^{14} &f_j=45,&\\
  &\ &\hspace{-.4cm}-2\leq &f_1\leq 0,\\
  &\ &\hspace{-.4cm}6\leq &f_j\leq 6, &j&= 2,\dots,5,\\
  &\ &\hspace{-.4cm}-1\leq &f_6\leq 0, \\
  &\ &\hspace{-.4cm}3\leq &f_j\leq 3, &j&=7,\dots,14.
\end{aligned} \right.
\end{align*} 
This is because only one choice of $f_i$s, namely  $f_1=0, f_2=f_3=f_4=f_5=6, f_6=0, f_7=f_8=f_9=f_{10}=f_{11}=f_{12}=f_{13}=f_{14}=3$,  satisfies the inequalities in the system, but $\sum\nolimits_{j=1}^{14} f_j=48\neq 45$. 
}\end{example}

Since $\rho'_2>\rho'_1$, either $\rho'_2>\rho'_1\geq 0$, or $\rho'_2\geq 0\geq \rho'_1$, or $0\geq \rho'_2>\rho'_1$. To solve \eqref{mastersys}, there are six cases to consider (see Table \ref{casestable}).  

\begin{table} 
		\centering
		\begin{tabular}{|l|l|}
			\hline
			$\begin{aligned}\\[1pt]
			&\iota_1=sm-\frac{sn}{2}-\frac{rm}{2}&\\
			&\iota_2=sm-\frac{sn}{2}&\text{if}\ \kappa_2\neq\emptyset\\
			&\rho_1=\frac{sm}{3}-\frac{rm}{3}&\\
			&\rho_2=\frac{sm}{3}&\text{if}\ \kappa_2\neq\emptyset\\
			&\rho'_1=\frac{sm}{2}-\frac{sn}{8}-\frac{3rm}{8}&\\
			&\rho'_2=\frac{sm}{2}-\frac{sn}{8}&\text{if}\ \kappa_2\neq\emptyset\\
			&\iota_{1j}= sm-\frac{sn}{4}-2\mult_j(u^3v)-\frac{3rm}{4}&\text{if}\ j\in \kappa_1\\
			&\iota_{2j}=sm-\frac{sn}{4}-2\mult_j(u^3v)&\text{if}\ j\in \kappa_2\\
			&\rho_{1j}=\frac{sm}{2}-\frac{3}{2}\mult_j(u^3v)-\frac{rm}{2}&\text{if}\ j\in \kappa_1\\
			&\rho_{2j}=\frac{sm}{2}-\frac{3}{2}\mult_j(u^3v)&\text{if}\ j\in \kappa_2\\[12pt]
			\end{aligned}$
			&
			$\begin{aligned}
			&\iota_1,\iota_2, \iota_{1j},\iota_{2j}\in \mathbb{Z}\\
			&\rho_1,\rho_2,\rho'_1,\rho'_2, \rho_{1j},\rho_{2j}\in \mathbb{R}
			\end{aligned}$\\
			
			\hline
			$\begin{aligned}\\[1pt]
			&\iota_2 >\iota_1&\\
			&\rho_2 >\rho_1\geq0&\\
			&\rho'_2 >\rho'_1&\\
			&\iota_1\leq\rho'_1\leq\rho_1&\\
			&\iota_2\leq\rho'_2\leq\rho_2&\text{if}\ \kappa_2\neq\emptyset \\[12pt]
			\end{aligned}$
			&
			$\begin{aligned}
			&\rho_1=0,\rho'_1<0&\text{if}&\ r=s\\
			&\iota_{1j}<0,\rho_{1j}=0&\text{if}&\ r=s
			\end{aligned}$\\
			\hline
			$\begin{aligned}
			&\iota_1\geq0\iff n\leq(2-\frac{r}{s})m\\
			&\iota_2\geq0\iff n\leq2m\\
			&\rho'_1\geq0\iff n\leq(4-\frac{3r}{s})m\\
			&\rho'_2\geq0\iff n\leq4m\\
			\end{aligned}$
			&
			$\begin{aligned}\\[1pt]
			&\iota_{1j}\geq 0 \iff \mult_j(u^3v)\leq \rho_1'&\text{if}\  j\in \kappa_1\\
			&\iota_{2j}\geq 0 \iff \mult_j(u^3v)\leq \rho_2'&\text{if}\  j\in\kappa_2\\
			&\rho_{1j}\geq 0 \iff \mult_j(u^3v)\leq \rho_1&\text{if}\  j\in \kappa_1\\
			&\rho_{2j}\geq 0 \iff \mult_j(u^3v)\leq \rho_2&\text{if}\  j\in \kappa_2\\
			&\rho_{1j}\geq \iota_{1j} \iff \mult_j(u^3v)\geq \iota_1&\text{if}\  j\in \kappa_1\\
			&\rho_{2j}\geq \iota_{2j} \iff \mult_j(u^3v)\geq \iota_2&\text{if}\  j\in \kappa_1\\[12pt]
			\end{aligned}$\\
			\hline
		\end{tabular}
		\vspace{.1cm}
		\caption{}
		\label{formulatable}
	\end{table}

\begingroup
\setlength{\tabcolsep}{10pt}      
\renewcommand{\arraystretch}{1.5}
\begin{table}
\centering
\begin{tabular}{|l|l|}
\hline
$\iota_1<\iota_2\leq0,\rho'_1<\rho'_2<0$&$n>4m$\\
\hline
$0\leq \iota_1<\iota_2,0\leq\rho'_1<\rho'_2$&$n\leq(2-r/s)m$\\
\hline
$\iota_1< 0\leq\iota_2,0\leq \rho'_1<\rho'_2$&$(2-r/s)m< n\leq\min\{2m,(4-3r/s)m\}$\\
\hline
$\iota_1<0<\iota_2,\rho'_1<0\leq \rho'_2$&$(4-3r/s)m< n<2m,s< 3r/2$\\
\hline
$\iota_1<\iota_2<0,0\leq \rho'_1<\rho'_2$&$2m< n\leq(4-3r/s)m,s>3r/2$\\
\hline
$\iota_1<\iota_2\leq0,\rho'_1<0\leq \rho'_2$&$\max\{2m,(4-3r/s)m\}< n\leq4m$\\
\hline
\end{tabular}
\vspace{.2cm}
\caption{}
\label{casestable}
\end{table}
\endgroup

\subsection{\texorpdfstring{$\bm{\iota_1<\iota_2\leq0,\rho'_1<\rho'_2<0}$}{}}
First, we color the $u^3v$-edges so that 
\begin{equation*}  
 \left \{ \begin{array}{ll}
\iota_1\leq \rho'_1< 0 \leq e_j\leq \rho_1 & \mbox { if } j\in\kappa_1,  \\
\iota_2\leq \rho'_2< 0\leq e_j\leq \rho_2 & \mbox { if } j\in\kappa_2. \end{array} \right.\\
\end{equation*}
By Lemma \ref{iota12leq}, $e \leq  q \lfloor \rho_1\rfloor+(k-q) \lfloor \rho_2\rfloor$, so this is possible. Then, we color the $u^2v^2$-edges so that 
\begin{equation*}  
 \left \{ \begin{array}{ll}
\iota_{1j}\leq 0 \leq f_j\leq \rho_{1j} & \mbox { if } j\in\kappa_1,  \\
\iota_{2j}\leq 0\leq f_j\leq \rho_{2j} & \mbox { if } j\in\kappa_2. \end{array} \right.\\
\end{equation*}
By Lemma \ref{rhoijlem}, $ f \leq   \sum\nolimits_{j\in \kappa_1} \lfloor \rho_{1j}\rfloor + \sum\nolimits_{j\in \kappa_2} \lfloor \rho_{2j}\rfloor$, so this is also  possible. 

\subsection{\texorpdfstring{$\bm{0\leq \iota_1<\iota_2,0\leq\rho'_1<\rho'_2}$}{}}
First, we color the $u^3v$-edges so that 
\begin{equation*}  
 \left \{ \begin{array}{ll}
 0\leq \iota_1 \leq e_j\leq \rho'_1\leq \rho_1 & \mbox { if } j\in\kappa_1,  \\
 0\leq \iota_2\leq e_j\leq \rho'_2\leq \rho_2 & \mbox { if } j\in\kappa_2. \end{array} \right.\\
\end{equation*}
By Lemmas \ref{iota12leq} and  \ref{lemu3vrhs}, $q\iota_1+(k-q)\iota_2\leq e \leq  q \lfloor \rho'_1\rfloor+(k-q) \lfloor \rho'_2\rfloor$, so this is possible. Then, we color the $u^2v^2$-edges so that 
\begin{equation*}  
 \left \{ \begin{array}{ll}
 0\leq\iota_{1j} \leq f_j\leq \rho_{1j} & \mbox { if } j\in\kappa_1,  \\
 0\leq\iota_{2j}\leq f_j\leq \rho_{2j} & \mbox { if } j\in\kappa_2. \end{array} \right.\\
\end{equation*}
By Lemmas \ref{iotaijlem} and \ref{rhoijlem}, $\sum\nolimits_{j\in \kappa_1}  \iota_{1j}+\sum\nolimits_{j\in \kappa_2}  \iota_{2j}\leq  f \leq   \sum\nolimits_{j\in \kappa_1} \lfloor \rho_{1j}\rfloor + \sum\nolimits_{j\in \kappa_2} \lfloor \rho_{2j}\rfloor$, so this is also  possible. 
\subsection{\texorpdfstring{$\bm{\iota_1< 0\leq\iota_2,0\leq \rho'_1<\rho'_2}$}{}}
First, we color the $u^3v$-edges so that 
\begin{equation*}  
 \left \{ \begin{array}{ll}
 \iota_1 <  0\leq e_j\leq \rho'_1\leq \rho_1 & \mbox { if } j\in\kappa_1,  \\
 0\leq \iota_2\leq e_j\leq \rho'_2\leq \rho_2 & \mbox { if } j\in\kappa_2. \end{array} \right.\\
\end{equation*}
By (N6) and Lemma   \ref{lemu3vrhs}, $(k-q)\iota_2\leq e \leq  q \lfloor \rho'_1\rfloor+(k-q) \lfloor \rho'_2\rfloor$, so this is possible.  Then, we color the $u^2v^2$-edges so that 
\begin{equation*}  
 \left \{ \begin{array}{ll}
 0\leq\iota_{1j} \leq f_j\leq \rho_{1j} & \mbox { if } j\in\kappa_1,  \\
 0\leq\iota_{2j}\leq f_j\leq \rho_{2j} & \mbox { if } j\in\kappa_2. \end{array} \right.\\
\end{equation*}
By Lemmas \ref{iotaijlem} and \ref{rhoijlem}, $\sum\nolimits_{j\in \kappa_1}  \iota_{1j}+\sum\nolimits_{j\in \kappa_2}  \iota_{2j}\leq  f \leq   \sum\nolimits_{j\in \kappa_1} \lfloor \rho_{1j}\rfloor + \sum\nolimits_{j\in \kappa_2} \lfloor \rho_{2j}\rfloor$, so this is also  possible. 

\subsection{\texorpdfstring{$\bm{\iota_1<0<\iota_2,\rho'_1<0\leq \rho'_2}$}{}}
First, we color the $u^3v$-edges so that
\begin{equation*}
  \left\{
  \begin{aligned}
    \iota_1 \leq \rho'_1<\  0= e_j&\leq \rho_1 & \mbox { if } j\in\kappa_1,  \\
     0 < \iota_2\leq e_j &\leq \rho'_2\leq  \rho_2 & \mbox { if } j\in\kappa_2.
  \end{aligned}
  \right.
\end{equation*}
By (N6) and Lemma  \ref{lemu3vrhs}, $(k-q)\iota_2\leq e \leq  (k-q)\left\lfloor\rho'_2\right\rfloor$, so this is possible.   Then, we color the $u^2v^2$-edges so that 
\begin{equation*}  
 \left \{ \begin{array}{ll}
  \iota_{1j} \leq0\leq f_j\leq \rho_{1j} & \mbox { if } j\in\kappa_1,  \\
 0\leq\iota_{2j}\leq f_j\leq \rho_{2j} & \mbox { if } j\in\kappa_2. \end{array} \right.\\
\end{equation*}
By Lemmas \ref{iotaijlem} and   \ref{rhoijlem}, $\sum\nolimits_{j\in \kappa_2}  \iota_{2j}\leq  f \leq   \sum\nolimits_{j\in \kappa_1} \lfloor \rho_{1j}\rfloor + \sum\nolimits_{j\in \kappa_2} \lfloor \rho_{2j}\rfloor$, so this is also  possible.

\subsection{\texorpdfstring{$\bm{\iota_1<\iota_2<0,0\leq \rho'_1<\rho'_2}$}{}} First, we show that for $i=1,2$, $\lceil \rho'_i\rceil  \leq \rho_i$.  Since $(m,r,\lambda)$ and $(n,s,\lambda)$ are admissible, we have $\lceil \rho'_i\rceil  \leq \rho'_i + 1/2$ for $i=1,2$. Moreover,  $m\geq4$, $s\geq  \lceil 3r/2 \rceil\geq 2$, and $n\geq 2m$. Therefore, 

\begin{align*}
24 (\rho_i-\lceil \rho'_i\rceil) &\geq 24 (\rho_i-\rho'_i - 1/2)\\
&=
\begin{dcases}
      3sn-4sm+rm-12 & \ \ \   \text{ if } i=1, \\
      3sn-4sm-12& \ \ \  \text{ if }i=2.
    \end{dcases}\\
    &\geq 3sn-4sm-12 \\&\geq 2sm-12>0.
\end{align*}
There are three cases to consider. 
\begin{enumerate}[label=(\roman*)]
\item $e\leq q\left\lfloor\rho'_1\right\rfloor+(k-q)\left\lfloor\rho'_2\right\rfloor$: First, we color the $u^3v$-edges so that
\begin{equation*}  
 \left \{ \begin{array}{ll}
 \iota_1 < 0\leq e_j\leq \rho'_1\leq \rho_1 & \mbox { if } j\in\kappa_1,  \\
 \iota_2 < 0\leq  e_j\leq \rho'_2\leq \rho_2 & \mbox { if } j\in\kappa_2. \end{array} \right.\\
\end{equation*}
By our hypothesis, $e \leq  q \lfloor \rho'_1\rfloor+(k-q) \lfloor \rho'_2\rfloor$, so this is possible. Then, we color the $u^2v^2$-edges so that 
\begin{equation*}  
 \left \{ \begin{array}{ll}
 0\leq \iota_{1j} \leq f_j\leq \rho_{1j} & \mbox { if } j\in\kappa_1,  \\
 0\leq\iota_{2j}\leq f_j\leq \rho_{2j} & \mbox { if } j\in\kappa_2. \end{array} \right.\\
\end{equation*}
By Lemmas \ref{iotaijlem} and \ref{rhoijlem}, $\sum\nolimits_{j\in \kappa_1}  \iota_{1j}+\sum\nolimits_{j\in \kappa_2}  \iota_{2j}\leq  f \leq   \sum\nolimits_{j\in \kappa_1} \lfloor \rho_{1j}\rfloor + \sum\nolimits_{j\in \kappa_2} \lfloor \rho_{2j}\rfloor$, so this is also  possible.  
\item $e\geq q\left\lceil\rho'_1\right\rceil+(k-q)\left\lceil\rho'_2\right\rceil$: First, we color the $u^3v$-edges so that
\begin{equation*}  
 \left \{ \begin{array}{ll}
\iota_1 < 0\leq   \rho'_1\leq e_j\leq \rho_1 & \mbox { if } j\in\kappa_1,  \\
 \iota_2<  0\leq\rho'_2\leq e_j\leq \rho_2 & \mbox { if } j\in\kappa_2. \end{array} \right.\\
\end{equation*}
By our hypothesis and  Lemma \ref{iota12leq}, $q\left\lceil\rho'_1\right\rceil+(k-q)\left\lceil\rho'_2\right\rceil\leq e \leq q \lfloor \rho_1\rfloor+(k-q) \lfloor \rho_2\rfloor $, so this is possible.  Then, we color the $u^2v^2$-edges so that 
\begin{equation*}  
 \left \{ \begin{array}{ll}
\iota_{1j} \leq 0\leq  f_j\leq \rho_{1j} & \mbox { if } j\in\kappa_1,  \\
\iota_{2j}\leq 0\leq f_j\leq \rho_{2j} & \mbox { if } j\in\kappa_2. \end{array} \right.\\
\end{equation*}
By Lemma  \ref{rhoijlem}, $f \leq   \sum\nolimits_{j\in \kappa_1} \lfloor \rho_{1j}\rfloor + \sum\nolimits_{j\in \kappa_2} \lfloor \rho_{2j}\rfloor$, so this is also  possible. 
\item $q\left\lfloor\rho'_1\right\rfloor+(k-q)\left\lfloor\rho'_2\right\rfloor<e<q\left\lceil\rho'_1\right\rceil+(k-q)\left\lceil\rho'_2\right\rceil$:
Observe that  at most one of $\rho'_1$ and $\rho'_2$ is an integer. 
We color the $u^3v$-edges so that
\begin{equation*}  
 \left \{ \begin{array}{ll}
 \iota_1 < 0\leq  \lfloor \rho'_1 \rfloor \leq e_j\leq \lceil \rho'_1\rceil  \leq \rho_1& \mbox { if } j\in\kappa_1,  \\
 \iota_2<  0\leq\lfloor \rho'_2 \rfloor \leq e_j\leq \lceil \rho'_2\rceil  \leq \rho_2 & \mbox { if } j\in\kappa_2. \end{array} \right.\\
\end{equation*}
By our hypothesis, $q\left\lfloor\rho'_1\right\rfloor+(k-q)\left\lfloor\rho'_2\right\rfloor<e<q\left\lceil\rho'_1\right\rceil+(k-q)\left\lceil\rho'_2\right\rceil$, so this is possible. Since
$$
\iota_{ij}=2(\rho'_i-e_j)\in \{-1,0, 1\} \mbox { for } i=1,2, j\in \kappa_i, 
$$
we have
\begin{equation} \label{plusminussum}
\sum\nolimits_{j\in J_1}  \iota_{1j}+\sum\nolimits_{j\in J_2}  \iota_{2j}  \leq k.
\end{equation}
Now, we color the $u^2v^2$-edges so that 
\begin{equation*}  
 \left \{ \begin{array}{ll}
\iota_{1j} \leq   f_j\leq \rho_{1j} & \mbox { if } j\in\kappa_1,  \\
\iota_{2j}\leq  f_j\leq \rho_{2j} & \mbox { if } j\in\kappa_2. \end{array} \right.\\
\end{equation*}
By \eqref{plusminussum}, and Lemmas  \ref{rhoijlem} and  \ref{M2Case3.3.3k-q+M2Case3.3.3k},  $$\sum\nolimits_{j\in J_1}  \iota_{1j}+\sum\nolimits_{j\in J_2}  \iota_{2j}  \leq k\leq  f \leq   \sum\nolimits_{j\in \kappa_1} \lfloor \rho_{1j}\rfloor + \sum\nolimits_{j\in \kappa_2} \lfloor \rho_{2j}\rfloor,$$ 
 so coloring the $u^2v^2$-edges is also possible.
\end{enumerate}

\subsection{\texorpdfstring{$\bm{\iota_1<\iota_2\leq0,\rho'_1<0\leq \rho'_2}$}{}} First, we show that $\lceil \rho'_2\rceil  \leq \rho_2$.  We have
$$z:=24 (\rho_2-\lceil \rho'_2\rceil)  = 24 (\rho_2-\rho'_2 - 1/2)=3sn-4sm-12$$
Since $n\geq \max\{2m,(4-3r/s)m\}$, we have 
\begin{align*}
n\geq  \begin{cases} 
      2m  & {\text{if}}\ s< 3r/2, \\
      (4-3r/s)m  & \text{if}\ s\geq 3r/2.
   \end{cases}
\end{align*}
Recall that if $m= 4$, then $r\geq 2$. If $m=5$, then since $(m,r,\lambda)$ is admissible, we have $r\geq 4$.  Therefore, if $n\geq  2m$, then $z\geq 2sm-12\geq 0$. If $n\geq (4-3r/s)m$ and  $s\geq 3r/2$, then
\begin{align*}
z  &\geq 3s\big ((4-3r/s)m\big )-4sm-12\\
    &=m (8s-9r) - 12\\
&\geq m (12r-9r) - 12\\
&\geq 3rm - 12\geq 0.
\end{align*}
There are three cases to consider. 
\begin{enumerate}[label=(\roman*)]
\item $e\leq (k-q)\left\lfloor\rho'_2\right\rfloor$: First, we color the $u^3v$-edges so that
\begin{equation*}
  \left\{
  \begin{aligned}
    \iota_1 \leq \rho'_1< 0= e_j&\leq \rho_1 & \mbox { if } j\in\kappa_1,  \\
    \iota_2\leq 0\leq  e_j&\leq \rho'_2\leq \rho_2 & \mbox { if } j\in\kappa_2.
  \end{aligned}
  \right.
\end{equation*}
This is possible by our hypothesis. Then, we color the $u^2v^2$-edges so that 
\begin{equation*}  
 \left \{ \begin{array}{ll}
  \iota_{1j} \leq0\leq f_j\leq \rho_{1j} & \mbox { if } j\in\kappa_1,  \\
 0\leq\iota_{2j}\leq f_j\leq \rho_{2j} & \mbox { if } j\in\kappa_2. \end{array} \right.\\
\end{equation*}
By  Lemmas \ref{iotaijlem} and \ref{rhoijlem}, $\sum\nolimits_{j\in \kappa_2}  \iota_{2j}\leq  f \leq   \sum\nolimits_{j\in \kappa_1} \lfloor \rho_{1j}\rfloor + \sum\nolimits_{j\in \kappa_2} \lfloor \rho_{2j}\rfloor$, so this is also  possible. 
\item $e\geq (k-q)\left\lceil\rho'_2\right\rceil$: We color the $u^3v$-edges such that
\begin{equation*}  
 \left \{ \begin{array}{ll}
 \iota_1\leq \rho'_1< 0\leq e_j\leq \rho_1 & \mbox { if } j\in\kappa_1,  \\
\iota_2\leq 0\leq \rho'_2\leq e_j\leq \rho_2 & \mbox { if } j\in\kappa_2. \end{array} \right.\\
\end{equation*}
By our hypothesis and  Lemma \ref{iota12leq}, $(k-q)\left\lceil\rho'_2\right\rceil\leq  e\leq q\left\lfloor\rho_1\right\rfloor +(k-q)\left\lfloor\rho_2\right\rfloor$, so this is   possible. Then, we color the $u^2v^2$-edges so that 
\begin{equation*}  
 \left \{ \begin{array}{ll}
  \iota_{1j} \leq0\leq f_j\leq \rho_{1j} & \mbox { if } j\in\kappa_1,  \\
\iota_{2j}\leq  0\leq f_j\leq \rho_{2j} & \mbox { if } j\in\kappa_2. \end{array} \right.\\
\end{equation*}
By Lemma  \ref{rhoijlem}, $f \leq   \sum\nolimits_{j\in \kappa_1} \lfloor \rho_{1j}\rfloor + \sum\nolimits_{j\in \kappa_2} \lfloor \rho_{2j}\rfloor$, so this is also  possible. 
\item $(k-q)\left\lfloor\rho'_2\right\rfloor<e<(k-q)\left\lceil\rho'_2\right\rceil$: Observe that $\rho'_2\notin \mathbb{N}$. We color the $u^3v$-edges such that
\begin{equation*}
  \left\{
  \begin{aligned}
    \iota_1\leq \rho'_1< 0=e_j &\leq \rho_1 & \mbox { if } j\in\kappa_1,  \\
    \iota_2\leq  0\leq\lfloor \rho'_2 \rfloor \leq e_j&\leq \lceil \rho'_2\rceil  \leq \rho_2 & \mbox { if } j\in\kappa_2.
  \end{aligned}
  \right.
\end{equation*}
By our hypothesis, $(k-q)\left\lfloor\rho'_2\right\rfloor<e<(k-q)\left\lceil\rho'_2\right\rceil$, so this is possible. Since
$$
\iota_{2j}=2(\rho'_2-e_j)\in \{-1, 1\} \mbox { for } j\in \kappa_2, 
$$
we have
\begin{equation} \label{plusminussum2}
\sum\nolimits_{j\in J_2}  \iota_{2j}  \leq k-q.
\end{equation}
Now, we color the $u^2v^2$-edges so that 
\begin{equation*}
  \left\{
  \begin{aligned}
    \iota_{1j} \leq0&\leq f_j\leq \rho_{1j} & \mbox { if } j\in\kappa_1,  \\
    \iota_{2j}&\leq f_j\leq \rho_{2j} & \mbox { if } j\in\kappa_2.
  \end{aligned}
  \right.
\end{equation*}
By \eqref{plusminussum2}, Lemmas \ref{rhoijlem} and  \ref{M2Case3.3.3k-q+M2Case3.3.3k}, $$\sum\nolimits_{j\in \kappa_2}  \iota_{2j}\leq k-q< k\leq  f \leq   \sum\nolimits_{j\in \kappa_1} \lfloor \rho_{1j}\rfloor + \sum\nolimits_{j\in \kappa_2} \lfloor \rho_{2j}\rfloor,$$  so coloring the $u^2v^2$-edges is also possible.
\end{enumerate}

\section{Inequalities}  \label{ineqsec} 
Before we can color the $u^3v$-edges  and $u^2v^2$-edges of $\cc F$, we need to establish several important inequalities. Most of these multivariate inequalities have been reduced to  two-variable inequalities and then  have been verified by the computer algebra system Mathematica. Since a central part of our argument relies on these inequalities, we shall provide  detailed proofs here. A reader is welcome to skip this section in the first reading. 

For $x\in \mathbb{R}$, $\frc(x)$ denotes the fractional part of $x$ which is $x-\lfloor x \rfloor$. 
\begin{lemma} \label{iota12leq} 
\begin{align*} 
q\iota_1 +(k-q)\iota_2  \leq e\leq q\lfloor\rho_1\rfloor +(k-q)\lfloor \rho_2\rfloor.
 \end{align*}
\end{lemma}
\begin{proof}
The following proves the left-hand side inequality.
\begin{align*}\label{longeq0} 
q\iota_1 +(k-q)\iota_2 & = ksm-\frac{ksn}{2}-\frac{qrm}{2}\\
&=\lambda\big(m-\frac{n}{2}\big)\binom{n-1}{3}-\lambda\frac{m}{2}\binom{m-1}{3}\\
&=\lambda m\binom{n-1}{3}- 2\lambda\binom{n}{4}-2\lambda\binom{m}{4}\\
&=\lambda(n-m)\binom{m}{3}-\frac{\lambda}{2}\binom{n-m}{3}(n+m-3)\\
&<\lambda(n-m)\binom{m}{3}.
\end{align*}
Since
\begin{align*}
3q\lfloor\rho_1\rfloor +3(k-q)\lfloor \rho_2\rfloor&= 3q\rho_1+3(k-q)\rho_2-3q\frc(\rho_1)-3(k-q)\frc(\rho_2)\\
    &=m(ks-qr)-3q\frc(\rho_1)-3(k-q)\frc(\rho_2)\\
&=\lambda m\left[\binom{n-1}{3}-\binom{m-1}{3}\right]-3q\frc(\rho_1)-3(k-q)\frc(\rho_2),
\end{align*}
the right hand side inequality  is equivalent to
$$\lambda m\binom{n-m}{3}+2\lambda\binom{m}{2}\binom{n-m}{2}\geq 3q\frc(\rho_1)+3(k-q)\frc(\rho_2).$$
There are three cases to consider.
\begin{enumerate} [label=({\alph*})]
\item $k>q, s>r$: In this case $s\geq 2, n\geq 4m/3$. Since  $\frc(\rho_i)\in \{0,1/3, 2/3\}$ (for $i=1,2$), we have $3q\frc(\rho_1)+3(k-q)\frc(\rho_2)\leq 2k\leq \lambda \binom{n-1}{3}$. Therefore, 
\begin{align*}
&\alpha:= \frac{6}{\lambda}\left[\lambda m\binom{n-m}{3}+2\lambda\binom{m}{2}\binom{n-m}{2}- 3q\frc(\rho_1)-3(k-q)\frc(\rho_2)\right]\\
&\qquad\geq6 m\binom{n-m}{3}+12\binom{m}{2}\binom{n-m}{2}-6\binom{n-1}{3}\\
&\qquad=(m-1)\left(n^3-6n^2+(-3m^2+6m+11)n+2m^3-m^2-6m-6\right).
\end{align*}
Let $f(n)=n^3-6n^2+(-3m^2+6m+11)n+2m^3-m^2-6m-6$, and $g(m)=10m^3-99m^2+234m-162$. Since $g(7)=55>0,g'(7)=318>0,g''(7)=222>0$, we have that $g(m)>0$ for $m\geq 7$, and so, $f(4m/3)=g(m)/27>0$ for $m\geq7$. Moreover, 
$f'(4m/3)=(7m^2-30m+33)/3>0$ and $f''(4m/3)=8m-12>0$ for $m\geq 7$. We conclude that $f(n)\geq0$ for $n\geq4m/3$ and $m\geq7$. Thus, $\alpha\geq 0$ for $m\geq 7, n\geq 4m/3$. 

Recall that if $m=4$, then $r\geq 2$. For $m=5,6$, since $(m,r,\lambda)$ is admissible, we have $4 \mid rm$. Consequently,  for  $m=4,5,6$, $s>r \geq 2$. We have
\begin{align*}
&\alpha\geq 6 m\binom{n-m}{3}+12\binom{m}{2}\binom{n-m}{2}-4\binom{n-1}{3}\\
&\qquad= \begin{cases} 
      \frac{1}{3}(10n^3-60n^2-106n+732) & {\text{if}}\ m=4, \\
      \frac{1}{3}(13n^3-78n^2-397n+2262) & {\text{if}}\ m=5, \\
      \frac{8}{3}(2n^3-12n^2-113n+663)  & \text{if}\ m=6.
   \end{cases}
\end{align*}
If $h(n)=10n^3-60n^2-106n+732, k(n)=13n^3-78n^2-397n+2262,\ \text{and}\ \ell(n)=2n^3-12n^2-113n+663$, then $h(6)=96>0,h'(6)=254>0, h''(6)=240>0 , k(7)=120>0, k'(7)=422>0, k''(7)=390>0, \ell(8)=15>0, \ell'(8)=79>0, \ell''(8)=72>0$. We conclude that  $k(n)\geq0$ for $n\geq7$ and $\ell(n)\geq0$ for $n\geq8$.
  Thus, $\alpha\geq 0$ for $m=4, n\geq6$, $m=5, n\geq 7$ and $m=6, n\geq 8$. 

\item $r=s$: In this case $n\geq 2m, \rho_1=0$. We have $3q\frc(\rho_1)+3(k-q)\frc(\rho_2)\leq 2(k-q)\leq 2\lambda \left[\binom{n-1}{3}-\binom{m-1}{3}\right]$. Therefore, 
\begin{align*}
&\frac{6}{\lambda}\left(\lambda m\binom{n-m}{3}+2\lambda\binom{m}{2}\binom{n-m}{2}- 3q\frc(\rho_1)-3(k-q)\frc(\rho_2)\right)\\
&\qquad\geq6 m\binom{n-m}{3}+12\binom{m}{2}\binom{n-m}{2}-12\binom{n-1}{3}+12\binom{m-1}{3}\\
&\qquad=(m-2)(n-m)(n^2 + m n- 6 n -2 m^2  - 3 m +  11)\\
&\qquad> n(n-6)-3m+11+m(n-2m)\\
&\qquad\geq  4n-3m+11>0.
\end{align*}
\item $k=q$: Since $3q\frc(\rho_1)+3(k-q)\frc(\rho_2)=3q\frc(\rho_1)\in \{0, q, 2q\}$, we have

\begin{align*}
\lambda m\binom{n-m}{3}+2\lambda\binom{m}{2}\binom{n-m}{2}\geq \begin{cases} 
      0 & {\text{if}}\ m(s-r)\equiv0\Mod{3}, \\
      q & {\text{if}}\ m(s-r)\equiv1\Mod{3}, \\
      2q  & \text{if}\ m(s-r)\equiv2\Mod{3},
   \end{cases}
\end{align*}
where the case $m(s-r)\equiv0\Mod{3}$ is trivial, and the remaining two cases are implied by (N8). 
\end{enumerate}
\end{proof}

\begin{lemma}\label{s_less_3r/2,k_not_q}
If $s<3r/2$, then $k>q$.
\end{lemma}
\begin{proof}
Suppose on the contrary that $k=q$, which implies $s=r\binom{n-1}{3}/\binom{m-1}{3}$. Since $s<3r/2$, we have $2\binom{n-1}{3}<3\binom{m-1}{3}$, or equivalently, $2(n-1)(n-2)(n-3)<3(m-1)(m-2)(m-3)$. Because $n\geq4m/3$, we have
\begin{align*}
&2(n-1)(n-2)(n-3)-3(m-1)(m-2)(m-3)\\
&\qquad\quad\geq \frac{1}{9}(32m^2-72m+36)(4m/3-3)-(3m^2-9m+6)(m-3)\\
&\qquad\quad\geq0,
\end{align*}
which is a contradiction. 
\end{proof}
\begin{lemma}\label{lemu3vrhs}
\begin{align*}
e\leq  \begin{cases} 
      q\lfloor\rho'_1\rfloor +(k-q)\lfloor \rho'_2\rfloor & {\text{if}}\  \iota_2\geq 0,\rho'_1\geq0,  \\
      (k-q)\left\lfloor\rho'_2\right\rfloor & \text{if}\  \iota_2> 0, \rho'_1< 0.
   \end{cases}
\end{align*}
\end{lemma}

\begin{proof}
First, let us assume that $\iota_2\geq0, \rho'_1\geq0$. In this case, $n\leq2m$. If $n<2m$, then $s>r$, and so $s\geq2$. Since $\frc(\rho'_i)\in \{0,1/2\}$ (for $i=1,2$), we have $q\frc(\rho'_1)+(k-q)\frc(\rho'_2)\leq k/2\leq \lambda \binom{n-1}{3}/4$.
We have 
\begin{align*}
2q\left\lfloor \rho'_1\right\rfloor+2(k-q)\left\lfloor\rho'_2\right\rfloor&=2q\rho'_1+2(k-q)\rho'_2-2q\frc(\rho'_1)-2(k-q)\frc(\rho'_2)\\
&=ksm-\frac{3qrm}{4}-\frac{ksn}{4}-2q\frc(\rho'_1)-2(k-q)\frc(\rho'_2)\\
&=\lambda m\binom{n-1}{3}-3\lambda \binom{m}{4}-\lambda\binom{n}{4}\\
&\ \ \ -2q\frc(\rho'_1)-2(k-q)\frc(\rho'_2)\\
&\geq \lambda m\binom{n-1}{3}-3\lambda \binom{m}{4}-\lambda\binom{n}{4}-\frac{\lambda}{2} \binom{n-1}{3},
\end{align*}
and hence, to prove the first inequality, it suffices to show that
$$\lambda m\binom{n-1}{3}-3\lambda \binom{m}{4}-\lambda\binom{n}{4}-\frac{\lambda}{2} \binom{n-1}{3} \geq 2 \lambda(n-m)\binom{m}{3}.$$

 Therefore, 

\begin{align*}
\alpha&:=\frac{12}{\lambda}\left[\lambda m\binom{n-1}{3}-3\lambda \binom{m}{4}-\lambda\binom{n}{4}-\frac{\lambda}{2} \binom{n-1}{3} - 2 \lambda(n-m)\binom{m}{3}\right]\\
&= 12m\binom{n-1}{3}-36\binom{m}{4}-12\binom{n}{4}-24(n-m)\binom{m}{3}-6\binom{n-1}{3}\\
&=12\binom{m}{2}\binom{n-m}{2}-6\binom{m-1}{3}-6(n-m)\binom{m-1}{2}\\
&\ \ \ -6(m-1)\binom{n-m}{2}-6\binom{n-m}{3}-12\binom{n-m}{4}\\
&=A+B+C+D,
\end{align*}
where $A=2\binom{m}{2}\binom{n-m}{2}-2\binom{n-m}{2}\binom{n-m-2}{2}$, $B=2\binom{m}{2}\binom{n-m}{2}-6(m-1)\binom{n-m}{2}$, $C=\binom{m}{2}\binom{n-m}{2}-2(n-m-2)\binom{n-m}{2}$, and $D=7\binom{m}{2}\binom{n-m}{2}-6(n-m)\binom{m-1}{2}-6\binom{m-1}{3}$. 

Now, suppose that $m\geq10$. Since $n\geq4m/3$, we have $n-m\geq m/3\geq4$. Since $n\leq 2m$, we have $\binom{m}{2}\geq \binom{n-m}{2}$, and so $A\geq0$. Since $m\geq6$, we have $4\binom{m}{2}\geq12(m-1)$, and hence, $B\geq0$. Moreover, $n\leq 2m$ implies that $\binom{m}{2}\geq 2(n-m-2)$, which means that $C\geq 0$. Finally, 
$$D=\frac{1}{2} (m - 1) \left(7 m n^2 +(- 14 m^2 - 19 m + 24) n+7 m^3  + 15 m^2   - 4 m  - 24\right).$$ Let $f(n)=7 m n^2 +(- 14 m^2 - 19 m + 24) n+7 m^3  + 15 m^2   - 4 m  - 24$ and $g(m)=7m^3-93m^2+252m-216$. Since $g(10)=4>0, g'(10)=492>0$, and $g''(10)=234>0$, we have $f(4m/3)=g(m)/9\geq0$, for $m\geq10$. Moreover, $f'(4m/3)=\frac{1}{3}(14m^2-57m+72)\geq0$, for $m\geq10$. Therefore, $D\geq0$, for $m\geq 10$, which implies that $\alpha\geq0$ for $m\geq10$.


For $m=7,8,9$, we have $\alpha\geq 24m\binom{n-1}{3}-72\binom{m}{4}-24\binom{n}{4}-48(n-m)\binom{m}{3}-12\binom{n-1}{3}$, which is nonnegative for $4m/3\leq n\leq 2m$.

For $m=4$, we have $s\geq3$. For $m=5,6$, since $(m,r,\lambda)$ is admissible, we have $4 \mid rm$. Consequently,  $r\geq 2$ and so $s\geq 3$. Hence, we have
$\alpha\geq 24m\binom{n-1}{3}-72\binom{m}{4}-24\binom{n}{4}-48(n-m)\binom{m}{3}-8\binom{n-1}{3}$, which is nonnegative for $m=4,5,6, 4m/3\leq n\leq 2m$. 

If $n=2m$, then $q\frc(\rho'_1)+(k-q)\frc(\rho'_2)\leq k/2\leq \lambda \binom{n-1}{3}/2$. Therefore, we need to show that 
$$\lambda m\binom{2m-1}{3}-3\lambda \binom{m}{4}-\lambda\binom{2m}{4}-\lambda \binom{2m-1}{3} \geq 2 \lambda(2m-m)\binom{m}{3},$$
which can be simplified as the inequality $(5 m - 3) (m - 1) (m - 2) (m - 4)/24\geq0$, which holds for $m\geq4$. 

Now, let us assume that $\iota_2\geq0, \rho'_1\leq 0$. In this case, $r<s< 3r/2$, and $(4-3r/s)m< n< 2m$. Therefore, $s\geq 4$, and $s/r< 3m/(4m-n)$. It is also clear that $r\leq\lambda\binom{m-1}{3}$. Since  $\frc(\rho'_2)\in\{0,1/2\}$, we have 
$$8(k-q)\frc(\rho'_2)\leq\frac{4\lambda}{s}\binom{n-1}{3}-\frac{4\lambda}{r}\binom{m-1}{3}\leq\lambda\binom{n-1}{3}-4.$$
Therefore, 
\begin{align*}
8(k-q)\left\lfloor\rho'_2\right\rfloor&=8(k-q)\rho'_2-8(k-q)\frc(\rho'_2)\\
&=4ksm-ksn-4qsm+qsn-8(k-q)\frc(\rho'_2)\\
&=(4m-n)(ks-qs)-8(k-q)\frc(\rho'_2)\\
&=\lambda(4m-n)\left[\binom{n-1}{3}-\frac{s}{r}\binom{m-1}{3}\right]-8(k-q)\frc(\rho'_2)\\
&\geq \lambda (4m-n)\left[\binom{n-1}{3}-\frac{s}{r}\binom{m-1}{3}\right]-\lambda\binom{n-1}{3}+4.
\end{align*}
Therefore, it suffices  to show that 
$$\lambda (4m-n)\left[\binom{n-1}{3}-\frac{s}{r}\binom{m-1}{3}\right]-\lambda\binom{n-1}{3}+4\geq8\lambda (n-m)\binom{m}{3}.$$
Hence,
\begin{align*}
\alpha&:= (4m-n)\left[\binom{n-1}{3}-\frac{s}{r}\binom{m-1}{3}\right]-8(n-m)\binom{m}{3}-\binom{n-1}{3}+4/\lambda \\
&\geq (4m-n)\binom{n-1}{3}-3m\binom{m-1}{3}-8(n-m)\binom{m}{3}-\binom{n-1}{3} \\
&\geq4\binom{m}{2}\binom{n-m}{2}-(n-m)\binom{m-1}{2}-4\binom{n-m}{4}\\
&\ \ \ -\binom{m-1}{3}-(m-1)\binom{n-m}{2}-\binom{n-m}{3}\\
&=A+B,
\end{align*}
where $A=3\binom{m}{2}\binom{n-m}{2}-(n-m)\binom{m-1}{2}$ and $B=\binom{m}{2}\binom{n-m}{2} -4\binom{n-m}{4}-\binom{m-1}{3}-(m-1)\binom{n-m}{2}-\binom{n-m}{3}$. Since $3\binom{n-m}{2}\geq n-m$ and $\binom{m}{2}\geq\binom{m-1}{2}$, $A$ is non-negative. We have $B=\binom{m-1}{2}\left[3\binom{n-m}{2}-m+3\right]/3-\binom{n-m}{2}\binom{n-m-1}{2}/3$. Since $n\leq 2m-1$, we have $\binom{m-1}{2}\geq \binom{n-m}{2}$. So, it is enough to show that $3\binom{n-m}{2}-m+3\geq \binom{n-m-1}{2}$, or equivalently, $n^2-2mn+m^2-m+2\geq0$. Let $f(n)=n^2-2mn+m^2-m+2$. Since $f(4m/3)=(m-3)(m-6)/9\geq0$ and $f'(4m/3)=2m/3\geq$ for $m\geq6$, we have $B\geq0$, for $m\geq6$. For $m=4,5$, we have
$$\alpha\geq (4m-n)\binom{n-1}{3}-3m\binom{m-1}{3}-8(n-m)\binom{m}{3}-\binom{n-1}{3},$$
which is nonnegative for $m=4,5$ and $4m/3\leq n\leq 2m-1$.

\end{proof}

\begin{lemma} \label{iotaijlem}
If  all the  $u^3v$-edges are colored, then 
\begin{align*}
f\geq  \begin{cases} 
      \sum\nolimits_{j\in\kappa_1}\iota_{1j}+\sum\nolimits_{j\in\kappa_2}\iota_{2j} \\
      \sum\nolimits_{j\in\kappa_2}\iota_{2j}  & \text{if}\ e_j=0 \ \text{for}\ j\in \kappa_1.
   \end{cases}
\end{align*}
\end{lemma}
\begin{proof}
We have
\begin{align*}
&\lambda\binom{m}{2}\binom{n-m}{2}- \sum\nolimits_{j\in\kappa_1}\iota_{1j}-\sum\nolimits_{j\in\kappa_2}\iota_{2j}\\
&\qquad= \lambda\binom{m}{2}\binom{n-m}{2}-k(sm-\frac{sn}{4})+q\frac{3rm}{4}+2\lambda(n-m)\binom{m}{3}\\
&\qquad= \lambda\binom{m}{2}\binom{n-m}{2}-\lambda m\binom{n-1}{3}+\lambda\binom{n}{4}+3\lambda\binom{m}{4}+2\lambda(n-m)\binom{m}{3}\\
&\qquad= \lambda\binom{n-m}{4}\geq 0.
\end{align*}

Now suppose that $e_j=0$ for $j\in \kappa_1$. We have
\begin{align*}
&\lambda\binom{m}{2}\binom{n-m}{2}-\sum\nolimits_{j\in\kappa_2}\iota_{2j}\\
&\qquad= \lambda\binom{m}{2}\binom{n-m}{2}-(k-q)(sm-\frac{sn}{4})+2\sum\nolimits_{j\in\kappa_2}e_j \\
&\qquad= \lambda\binom{m}{2}\binom{n-m}{2}+2\lambda(n-m)\binom{m}{3}-(k-q)(sm-\frac{sn}{4})\geq 0,
\end{align*}
where the last inequality holds by (N7). 
\end{proof}

\begin{lemma} \label{rhoijlem}
 We can color the $u^3v$-edges so that
$$f\leq \sum\nolimits_{j\in\kappa_1}\rounddown{\rho_{1j}}+\sum\nolimits_{j\in\kappa_2}\rounddown{\rho_{2j}}.$$ 
\end{lemma}
\begin{proof}
We prove something stronger. We show that in most cases as long as  all the  $u^3v$-edges are colored, the above inequality holds. In the remaining few cases, we color the $u^3v$-edges carefully ensuring that certain conditions are met  (see Table \ref{smallcasestable}).  Since $(m,r,\lambda)$ is admissible, $4 \mid rm$. Therefore, 
\begin{equation*}
\frc(\rho_{ij})=\left\{\!\begin{aligned}
 0 & \ \ {\text{if }}  sm+\mult_j(u^3v) \equiv 0 \Mod 2 \\
 \frac{1}{2} & \ \ \text{if } sm+\mult_j(u^3v) \equiv 1 \Mod 2
\end{aligned}\right\}
\quad j\in \kappa_i, i=1,2.
\end{equation*}
Let $\beta=\sum\nolimits_{j\in\kappa_1}\frc(\rho_{1j})+\sum\nolimits_{j\in\kappa_2}\frc(\rho_{2j})$. Then 
$$2\beta= |\{j\in \kappa:sm+\mult_j(u^3v) \equiv 1 \Mod 2 \}|\leq k$$
We have
\begin{align*}
&2\sum\nolimits_{j\in\kappa_1}\lfloor \rho_{1j}\rfloor+2\sum\nolimits_{j\in\kappa_2}\lfloor \rho_{2j}\rfloor-2\lambda\binom{m}{2}\binom{n-m}{2}\\
&\qquad=2\sum\nolimits_{j\in\kappa_1} \rho_{1j}+2\sum\nolimits_{j\in\kappa_2}\rho_{2j}-2\lambda\binom{m}{2}\binom{n-m}{2}-2\beta\\
&\qquad=qm(s-r)+(k-q)sm-3\lambda(n-m)\binom{m}{3}-2\lambda\binom{m}{2}\binom{n-m}{2}-2\beta\\
&\qquad= ksm-qrm-3\lambda(n-m)\binom{m}{3}-2\lambda\binom{m}{2}\binom{n-m}{2}-2\beta\\
&\qquad= \lambda m\left[\binom{n-1}{3}-\binom{m-1}{3}\right]-3\lambda(n-m)\binom{m}{3}-2\lambda\binom{m}{2}\binom{n-m}{2}-2\beta\\
&\qquad= \lambda m\binom{n-m}{3}-2\beta.
\end{align*}

First, let us assume that $n\geq3m$. We have

\begin{align*}
\alpha&:=\frac{3}{\lambda}\left[\lambda m\binom{n-m}{3}-2\beta\right]\\
&\geq \frac{3}{\lambda}\left[\lambda m\binom{n-m}{3}-k\right]\\
&\geq 3m\binom{n-m}{3}-3\binom{n-1}{3}\\
&=3m\binom{n-m}{3}-3\binom{m-1}{3}-3(n-m)\binom{m-1}{2}-3(m-1)\binom{n-m}{2}-3\binom{n-m}{3}\\
&=(m-1)(n-m-5)\binom{n-m}{2}-(4m-3n-3)\binom{m-1}{2}.
\end{align*}

Since $n\geq3m$, we have $\binom{n-m}{2}\geq\binom{m-1}{2}$ and $(m-1)(n-m-5)\geq 4m-3n-3$. This implies $\alpha\geq0$.

Now, let us assume that $4m/3\leq n\leq 3m, s\geq2,\ \text{and}\ m\geq38$. Since $s\geq2$, we have $2k\leq \lambda\binom{n-1}{3}$. Therefore,

\begin{align*}
4\alpha&\geq 12m\binom{n-m}{3}-6\binom{n-1}{3}\\
&=12m\binom{n-m}{3}-6\left[\binom{m-1}{3}+(n-m)\binom{m-1}{2}+(m-1)\binom{n-m}{2}+\binom{n-m}{3}\right]\\
&=6(2m-1)\binom{n-m}{3}-6\binom{m-1}{3}-6(n-m)\binom{m-1}{2}-6(m-1)\binom{n-m}{2}\\
&\geq (m-1)\left(2n^3-3(2m+3)n^2+(6m^2+15m+13)n-2m^3-7m^2-8m-6\right).
\end{align*}

Let $f(n)=2n^3-3(2m+3)n^2+(6m^2+15m+13)n-2m^3-7m^2-8m-6$, $g(m)=2m^3-81m^2+252m-162$, and $h(m)= 2 m^3 - 13 m^2 + 18 m - 6$. Since $g(38)=2194>0,g'(38)=2760>0,g''(38)=294>0$, we have that $g(m)>0$ for $m\geq 38$, and so, $f(4m/3)=g(m)/27>0$ for $m\geq38$. Moreover, 
$f'(4m/3)=(2m^2-27m+39)/3>0$ and $f''(4m/3)=4m-18>0$ for $m\geq 38$. We conclude that $f(n)\geq0$ for $n\geq4m/3$ and $m\geq38$. Thus, $\alpha\geq 0$ for $m\geq 38, n\geq 4m/3$. 

Now, let us assume that $4m/3\leq n\leq3m,r=s=1,\ \text{and}\ m\geq29$. Since $r=s$, we have $n\geq2m$. Therefore, 

\begin{align*}
\alpha&\geq \frac{6}{\lambda}\left[\lambda m\binom{n-m}{3}-k\right]\\
&\geq 6m\binom{n-m}{3}-6\binom{n-1}{3}\\
&\geq 6m\binom{m}{3}-6\binom{3m-1}{3}\\
&=(m-1)[m^2(m-2)-3(3m-1)(3m-2)]\\
&\geq(m-1)[m^2(m-2)-9m(3m-2)]\\
&=m(m-1)(m^2-29m+18)\\
&\geq0.
\end{align*}

For the remaining  cases, we refer the reader to  Table \ref{smallcasestable}.

\end{proof}
\begin{table}[h!] 
\centering
 \begin{adjustbox}{angle=90}
\begin{tabular}{|c|c|c|c|c|c|c|c|c|c|c|c|c|c|c|c|c|c|c|}
\hline
\multicolumn{1}{|c|}{\multirow{2}{*}{$m$}} & \multicolumn{1}{c|}{\multirow{2}{*}{$n$}} & \multicolumn{1}{c|}{\multirow{2}{*}{$r$}} & \multicolumn{1}{c|}{\multirow{2}{*}{$s$}} & \multicolumn{1}{c|}{\multirow{2}{*}{$q$}} & \multicolumn{1}{c|}{\multirow{2}{*}{$k$}} & \multicolumn{1}{c|}{\multirow{2}{*}{$k-q$}} & \multicolumn{1}{c|}{\multirow{2}{*}{$\iota_1$}} & \multicolumn{1}{c|}{\multirow{2}{*}{$\left\lfloor\rho'_1\right\rfloor$}}& \multicolumn{1}{c|}{\multirow{2}{*}{$\left\lfloor\rho_1\right\rfloor$}} & \multicolumn{1}{c|}{\multirow{2}{*}{$\iota_2$}} & \multicolumn{1}{c|}{\multirow{2}{*}{$\left\lfloor\rho'_2\right\rfloor$}} & \multicolumn{1}{c|}{\multirow{2}{*}{$\left\lfloor\rho_2\right\rfloor$}} & \multicolumn{1}{c|}{\multirow{2}{*}{$e$}} & \multicolumn{1}{c|}{\multirow{2}{*}{$f$}} & \multicolumn{1}{c|}{\multirow{2}{*}{$g$}} & \multicolumn{1}{c|}{\multirow{2}{*}{Case}} & \multicolumn{2}{c|}{$e_j$}\\
 \cline{18-19}
 \multicolumn{1}{|c|}{} &&&&&&&&&&&&&&&&&$j\in\kappa_1$&$j\in\kappa_2$\\
 \hline
5& 8& 4& 5& 1&  7&6& -5&  0&1& 5& 7 & 8 & 30& 30& 5&5.3&$0^1$&$5^6$\\
\hline
6 & 8 & 2 & 5 & 5 & 7  &2 & 4 & 5 & 6 & 10 & 10 & 10 & 40 & 15 & 0 & 5.2       & $4^5$ & $10^2$\\
\hline
6 & 9 & 2 & 4 & 5 & 14 &9& 0 & 3 & 4 & 6 & 7 & 8 & 60 & 45 & 6 & 5.2 & $0^2,2^3$ & $6^9$ \\
\hline
6& 9 & 2 & 8 & 5 & 7 &2& 6 & 10 & 12 & 12 & 15 & 16 & 60 & 45 & 6 & 5.2      & $6^2,8^3$ & $12^2$\\
\hline
8 & 12 & 1 & 3 & 35 & 55 &20& 2 & 4 & 5 & 6 & 7 & 8 & 224 & 168 & 32 & 5.2 & $2^{18},4^{17}$ & $6^{20}$\\
\hline
8 & 16 & 1 & 1 & 35 & 455 & 420 & -4 & -1 & 0 & 0 & 2 & 2 & 448 & 784 & 448 &5.6(i)&$0^{35}$ & $0^{196},2^{224}$\\
\hline
8 & 11 & 5 & 8 & 7 & 15 & 8 & 0 & 6 & 8 & 20 & 21 & 21 & 168 & 84 & 8 &5.2&$0^3,2^4$&$20^8$\\
\hline
8 & 11 & 5 & 12 & 7 & 10 & 3 & 10 & 16 & 18 & 30 & 31 & 32 & 168 & 84 & 8 & 5.2 &$10^3,12^4$ & $30^3$\\
\hline
8 & 11 & 7 & 12 & 5 & 10 & 5& 2 & 10 & 13 & 30 & 31 & 32 & 168 & 84 & 8 & 5.2 & $2^1,4^4$ & $30^5$ \\
\hline
9& 12& 4& 11 & 14 & 15& 1&15 & 19 & 21 & 33 & 33 & 33 & 252 & 108 & 9&5.2&$15^{11},18^3$&$33^1$\\
\hline
9 & 12 & 8& 15 & 7 & 11& 4&9& 18& 21& 45&  45 & 45 & 252 & 108 & 9&5.2&$9^6,18^1$&$45^4$\\
\hline
12&  18&  1 & 2 & 165 & 340& 175 & 0& 3& 4& 6& 7& 8& 1320 & 990 & 240&5.2&$0^{43},2^{121},3^1$&$6^{150},7^{25}$\\
\hline
12& 16&  3 & 5&  55&  91& 36 & 2 & 6& 8 & 20 & 20&  20&  880 & 396 &  48&5.2&$2^{30},4^{25}$&$20^{36}$\\
\hline
12&  16&  3& 7&  55&  65& 10 & 10&  14& 16&  28&  28 & 28&  880 & 396 & 48&5.2&$10^{30},12^{25}$&$28^{10}$\\
\hline
14& 19 & 2 & 4& 143 &204 & 61 & 4& 8&  $9$ & 18 & 18 & 18 & 1820 & 910 &140&5.2&$4^{68},6^{75}$&$18^{61}$\\
\hline
14&  20&  2&  3&  143&  323 & 180 & -2& 3&  4 & 12& 13& 14 & 2184 & 1365 & 280&5.3&$0^{131},2^{12}$&$12^{180}$\\
\hline
16 & 22&  1&  2&  455&  665& 210 & 2 & 4 & 5 & 10&  10 & 10& 3360 & 1800 &  320&5.2&$2^{280},4^{175}$&$10^{210}$\\
\hline
28 & 38&  1&  2&  2925 & 3885 & 960 & 4 & 8& 9 & 18 & 18 &  18& 32760 & 17010 &  3360&5.2&$4^{1035},6^{1890}$&$18^{960}$\\
\hline
 5&  7 & 4& 20 & 1&  1&0& 20 & 25 & 26 &NA&NA&NA& 20& 10& 0&5.2&$20^1$&NA\\
\hline
 6& 8& 2& 7& 5& 5& 0& 8&  9 &10 &NA&NA&NA& 40&  15&  0&5.2&$8^5$&NA\\
\hline
6& 8& 10&  35&  1&  1&0&  40& 47&  50 &NA &NA&NA&40& 15& 0&5.2&$40^1$&NA\\
\hline
\end{tabular}
 \end{adjustbox}
\vspace{.1cm}
\caption{}
 \label{smallcasestable}
\end{table}

\begin{lemma}\label{M2Case3.3.3k-q+M2Case3.3.3k}
If $\iota_2\leq0$ and  $\rho'_2\geq0$, then $f\geq k$.
\end{lemma}
\begin{proof}
 Since $\iota_2\leq0$ and $\rho'_2\geq0$, we have $2m\leq n\leq4m$. Now, we have 
\begin{align*}
\frac{12}{\lambda}\left[\lambda\binom{m}{2}\binom{n-m}{2}-k\right]&=12\binom{m}{2}\binom{n-m}{2}-\frac{12}{s}\binom{n-1}{3}\\
&\geq12\binom{m}{2}\binom{n-m}{2}-12\binom{n-1}{3}\\
&=-2 n^3+( 3 m^2 - 3 m+12) n^2 +(-6 m^3+ 3 m^2 + 3 m-22) n\\
&\ \ \  +3 m^4-3 m^2  + 12.
\end{align*}
Let $f(n)=-2 n^3+( 3 m^2 - 3 m+12) n^2 +(-6 m^3+ 3 m^2 + 3 m-22) n +3 m^4-3 m^2  + 12, g(m)=3 m^2 - 13 m + 6, h(m)=27 m^3 - 164 m^2 + 201 m - 88, k(m)=6 m^3 - 33 m^2 + 51 m - 22, l(m)=18 m^3 - 117 m^2 + 99 m - 22$. Since $g(6)=36>0, g'(6)=23>0, h(6)=1046>0, h'(6)=1149>0, h''(6)=644>0, k(6)=392>0, k'(6)=303>0, k''(6)=150>0, l(6)=248>0, l'(6)=639>0, l''(6)=414>0$, we have $f(2m)=(m-1)(m-2)g(m)\geq0, f'(2m)=k(m)\geq0,  f(4m)=27m^4-164m^3+201m^2-88m+12>mh(m)\geq0,\ \text{and}\ f'(4m)=l(m)\geq0$ for $m\geq6$.

If $m=5$, then by \eqref{divcond+}, $4\mid 5r$. Hence, $r\geq4$, which implies $s\geq4$. Moreover, if $m=4$, we have $r\geq2$, which implies $s\geq2$. Hence, we have

\begin{align*}
\frac{12}{\lambda}\left[\lambda\binom{m}{2}\binom{n-m}{2}-k\right]&\geq 12\binom{m}{2}\binom{n-m}{2}-6\binom{n-1}{3}\\
&=
	\begin{cases}
	-n^3+42n^2-335n+726 & \text{if}\ m=4\\
	-n^3+66n^2-671n+1806 & \text{if}\ m=5.
	\end{cases}
\end{align*}
Let $f(n)=-n^3+42n^2-335n+726, g(n)=-n^3+66n^2-671n+1806$. Since $f(8)=222>0, f'(8)=145>0, f(16)=2022>0, f'(16)=241>0, g(10)=696>0, g'(10)=349>0,g(20)=6786>0,g'(20)=1849>0$, we have $\lambda\binom{m}{2}\binom{n-m}{2}-k\geq0$, for $m=4,5, 2m\leq n\leq 4m$. This completes the proof.

\end{proof}

\section*{Acknowledgement}
 The authors wish to thank Anna Johnsen, Lana K\"uhle, Stefan Napirata, Songling Shan, and  Lisa Zyga for their  feedback on various aspects of this paper.

\bibliographystyle{plain}
\bibliography{rKmh4rs}
\end{document}